\documentclass[12pt]{article}
\usepackage{amsmath,amssymb,amsthm,mathtools,bm}
\usepackage{ytableau}
\usepackage[pdfencoding=auto]{hyperref}

\title{Algorithmically Distinguishing Irreducible Characters of the Symmetric Group}
\author{Timothy Y. Chow \and Jennifer Paulhus}
\date{August 2020}

\newcommand{\dop}{doppelg\"anger}
\newcommand{\Dop}{Doppelg\"anger}
\DeclareMathOperator{\sgn}{sgn}

\newtheorem{theorem}{Theorem}
\newtheorem{lemma}{Lemma}

\theoremstyle{definition}
\newtheorem{definition}{Definition}

\begin{document}
\maketitle

\begin{abstract} 
Suppose that $\chi_\lambda$ and $\chi_\mu$ are
distinct irreducible characters of the symmetric group $S_n$.
We give an algorithm that, in time polynomial in~$n$,
constructs $\pi\in S_n$
such that $\chi_\lambda(\pi)$ is provably different from $\chi_\mu(\pi)$.
In fact, we show a little more.
Suppose $f = \chi_\lambda$ for some irreducible character
$\chi_\lambda$ of $S_n$,
but we do not know~$\lambda$,
and we are given only oracle access to $f$.
We give an algorithm that determines~$\lambda$,
using a number of queries to $f$ that is polynomial in~$n$.
Each query can be computed in time polynomial in~$n$
by someone who knows~$\lambda$.
\end{abstract}

\section{Introduction}

This paper is motivated by the following question.
Suppose that we are given
two distinct irreducible characters
$\chi_\lambda$ and~$\chi_\mu$
of the symmetric group~$S_n$.
How hard is it to find a permutation $\pi\in S_n$
such that $\chi_\lambda(\pi) \ne \chi_\mu(\pi)$?

Surprisingly, this simple and natural question does
not seem to have been considered before in the literature.
On the one hand, one might guess that the problem is hard,
since Pak and Panova~\cite[Theorem 7.1]{pak-panova}
have shown that even determining whether
$\chi_\lambda(\pi) = 0$ is in general
$\mathsf{NP}$-hard---in fact,
it is \emph{strongly} $\mathsf{NP}$-hard
(Pak, personal communication),
meaning there is no algorithm that runs in time polynomial in~$n$
unless $\mathsf{P} = \mathsf{NP}$.

On the other hand, empirically,
if one simply tries various permutations---especially
permutations with a lot of fixed points---then
it seems to take at most a few tries to find a~$\pi$
such that $\chi_\lambda(\pi) \ne \chi_\mu(\pi)$.
However, proving that this heuristic procedure always works
does not seem to be easy.
For example, Craven~\cite{craven} has shown that the number of
distinct irreducible characters of~$S_n$ with the same degree
can be arbitrarily large.
Similarly, it seems that
known results on character values (e.g.,~\cite{larsen-shalev})
do not provide us with enough control over
``unexpected'' equalities of the form
$\chi_\lambda(\pi)  = \chi_\mu(\pi)$
(for $\pi$ with many fixed points) to answer our question.

Nevertheless, in this paper we give an algorithm that
solves the stated problem in polynomial time.
The heart of our solution is an algorithm for
the following related problem.
We are given a positive integer~$n$
as well as \emph{oracle access}
to a function $f$ on the symmetric group~$S_n$,
meaning that the only way we can obtain information about~$f$
is to submit a \emph{query} (i.e., an input value that we are free to choose)
$\pi\in S_n$ to an oracle,
which then truthfully tells us the value of~$f(\pi)$.
We are promised that $f = \chi_\lambda$ for some
irreducible character $\chi_\lambda$ of~$S_n$,
but we do not know~$\lambda$.
Our job is to determine~$\lambda$
via a sequence of queries to the oracle.
Our queries are allowed to be \emph{adaptive}; that is,
we may examine the results of previous queries
when deciding which query to submit next.

\begin{theorem}
\label{thm:main}
There is a deterministic algorithm that, given oracle access to
a function~$f$ that is promised to be an irreducible
character of~$S_n$, determines which irreducible character it is,
using a number of queries that is polynomial in~$n$.
\end{theorem}

Note that Theorem~\ref{thm:main} focuses not on
\emph{computational complexity} but on
\emph{query complexity}, since the latter is more natural
in a context where the irreducible character is unknown.
However, an interesting feature of our algorithm is that,
instead of querying permutations with many fixed points,
it mainly queries permutations with rather long cycles
and very few fixed points.
The possible border-strip tableaux are thereby
severely constrained, allowing us to enumerate them explicitly
and prove the inequalities we want.

To solve our original problem
of finding~$\pi$ such that $\chi_\lambda(\pi) \ne \chi_\mu(\pi)$,
one simply simulates the algorithm using $f = \chi_\lambda$
until one reaches a query that rules out the possibility
that $f=\chi_\mu$.
As will become apparent when we describe the algorithm,
the specific computations we need can all be done in time polynomial
in~$n$ when $\lambda$ is known;
in fact, in many cases,
all that is needed is to determine whether $\chi_\lambda(\pi)$
is nonzero, or whether it is even or odd.
Though, as we noted above, such questions can be hard in general,
they are easy in the cases we need.

After some necessary preliminaries in Section~\ref{sec:background},
we describe the overall structure of our algorithm
in Section~\ref{sec:sketch}.
There is one step of the algorithm that,
as far as we can see, requires a complicated
case analysis; this is carried out in Section~\ref{sec:firsthook}.

\section{Background}
\label{sec:background}

In this section we review some standard material.
It turns out that in the case of the symmetric group,
there is a natural bijection
between irreducible characters
and conjugacy classes, and conjugacy classes are
naturally indexed by partitions (the lengths of the cycles
of the permutation).
So the first order of business is to review
some of the combinatorics of partitions.

\subsection{Partitions and Young Diagrams}
\label{sec:prelim}

Let $n$ be a positive integer.
A \emph{composition of~$n$} is defined to be
a sequence $\alpha = (\alpha_1, \ldots, \alpha_\ell)$
of positive integers such that $\sum_{i=1}^\ell \alpha_i= n$.
A \emph{partition of~$n$} is a composition
$\lambda = (\lambda_1, \ldots, \lambda_\ell)$ of~$n$ such that
\begin{equation*}
\lambda_1 \ge \lambda_2 \ge \cdots \ge \lambda_\ell.
\end{equation*}
A partition of~$n$ may be visualized as a \emph{Young diagram},
which is a left-justified grid of boxes having $\lambda_i$ boxes
in row~$i$.  See Figure~\ref{fig:youngdiagram} for an example.
The boxes of a Young diagram are coordinatized
in the same way that matrix entries are coordinatized;
i.e., box $(i,j)$ is the $j$th box from the left
in the $i$th row from the top, where $i\ge1$ and $j\ge1$.

\ytableausetup{centertableaux,smalltableaux}
\begin{figure}[!ht]
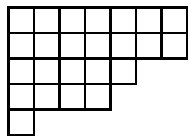

\begin{center}
 \begin{ytableau}
  \   &       &       &       &     &   & \\
      &       &       &       &     &   & \\
      &       &       &       &       \\
      &       &       &               \\
  \ 
\end{ytableau}
\end{center}
\caption{Young Diagram of the Partition $(7,7,5,4,1)$}
\label{fig:youngdiagram}
\end{figure}

Following standard terminology~\cite[Section 7.2]{stanley2},
we define the \emph{conjugate~$\lambda'$} of a partition~$\lambda$
to be the sequence of column lengths of
the Young diagram of~$\lambda$.
For example, the conjugate of $(7,7,5,4,1)$ is $(5,4,4,4,3,2,2)$.
A partition~$\lambda$ is \emph{self-conjugate} if $\lambda=\lambda'$.

The \emph{principal diagonal} of a Young diagram is
the set of boxes with coordinates $(i,i)$ for some~$i$.

For the purposes of this paper, it will be convenient
to think of Young diagrams in a slightly nonstandard manner,
namely as a nested sequence of \emph{principal hooks}.

\begin{definition}

The \emph{$i$th principal hook} of a Young diagram~$D$ is the set
\begin{equation*}
H_i := \{(i,j)\in D : j\ge i\} \cup \{(j,i)\in D : j\ge i\}.
\end{equation*}
The \emph{$i$th principal hook length} $h_i$ is the
area (i.e., cardinality) of~$H_i$.
\end{definition}

 For example, Figure~\ref{fig:principalhooks} illustrates
the (nonempty) principal hooks of
the Young diagram from Figure~\ref{fig:youngdiagram},
where we have colored $H_1$ red, $H_2$ orange, $H_3$ yellow, and $H_4$ green.

\begin{figure}[!ht]
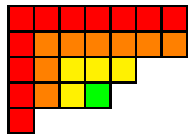

\begin{center}
\begin{ytableau}
  *(red)   &  *(red)     &  *(red)  &  *(red)   & *(red)  & *(red)  & *(red)  \\
  *(red)   &  *(orange)  &  *(orange) & *(orange) & *(orange) & *(orange)  & *(orange) \\
  *(red)   &  *(orange)  &  *(yellow) & *(yellow) & *(yellow)  \\
  *(red)   &  *(orange)  &  *(yellow) & *(green) \\
  *(red)
\end{ytableau}
\end{center}
\caption{Principal Hook Decomposition of $(7,7,5,4,1)$}
\label{fig:principalhooks}
\end{figure}

 The horizontal part of a principal hook is known as its
\emph{arm} and the vertical part is known as its \emph{leg}.
Note that the arm of $H_i$ extends at least as far to the right
as the arm of $H_{i+1}$, and may extend farther.
If it extends farther, we refer to the extra boxes as
the \emph{$i$th arm overhang} (and similarly for the legs).
More formally, we have the following definitions.

\begin{definition}

Suppose that a Young diagram has
exactly $k$ nonempty principal hooks.
Then for $1 \le i \le k$, the \emph{$i$th arm overhang} is
\begin{equation*}
\{(i,j) \in H_i : \mbox{$j > i$ and $(i+1,j) \notin H_{i+1}$} \}.
\end{equation*}
Similarly, the \emph{$i$th leg overhang} is
\begin{equation*}
\{(j,i) \in H_i : \mbox{$j > i$ and $(j,i+1) \notin H_{i+1}$} \}.
\end{equation*}
If the $i$th arm overhang and the $i$th leg overhang have
different cardinalities, then we call the smaller one the
\emph{$i$th short overhang} and we let $a_i$ be its cardinality;
similarly we call the longer one the \emph{$i$th long overhang}
and we let $b_i$ be its cardinality.  If the $i$th arm overhang
and the $i$th leg overhang have the same cardinality then
we set both $a_i$ and $b_i$ equal to that cardinality.
\end{definition}

 In our running example, the 1st leg overhang
and the 2nd and 3rd arm overhangs are nonempty,
as indicated by the colored boxes in Figure~\ref{fig:overhangs},
and $a_1 = a_2 = a_3 = a_4 = 0$ and
$b_1 = 1$, $b_2 = 2$, $b_3 = 1$, $b_4 = 0$.

\begin{figure}[!ht]
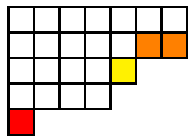

\begin{center}
\begin{ytableau}
 \   &  &  & & & &   \\
     &  &  & & & *(orange) & *(orange) \\
     &  &  &  & *(yellow)  \\
     &  &  & \\
 *(red) 
\end{ytableau}
\end{center}
\caption{Nonempty Overhangs of $(7,7,5,4,1)$}
\label{fig:overhangs}
\end{figure}

\begin{lemma}
\label{lem:hkak3}
If there are more than $k$ principal hooks, then $h_k - a_k \ge 3$.
\end{lemma}

\begin{proof}

By assumption, the $(k+1)$st principal hook has at least one box,
so the $k$th principal hook must have at least three boxes
that do not overhang the $(k+1)$st principal hook.
\end{proof}

The next definition plays a key role in this paper.

\begin{definition}
The \emph{\dop} $\hat\lambda$ of a partition $\lambda$ is
the partition whose Young diagram~$\hat D$
is the same as the Young diagram~$D$ of~$\lambda$
except that the cardinality of the 1st arm overhang of~$\hat\lambda$
equals the cardinality of the 1st leg overhang of~$\lambda$,
and vice versa.
\end{definition}

 For example, the \dop\ of $(7,7,5,4,1)$
is $(8,7,5,4)$, obtained by swapping the 1st leg overhang
(whose cardinality is~$1$) with the 1st arm overhang
(whose cardinality is~$0$).
Note that $\lambda$ and~$\hat\lambda$ have
the same values of $h_i$, $a_i$, and~$b_i$ for all~$i$.

The reason we have chosen the term ``\dop'' is that it turns out
to be surprisingly tricky to find a permutation~$\pi$ such that
$\chi_\lambda(\pi)$ is provably different from $\chi_{\hat\lambda}(\pi)$.

\subsection{Border-Strip Tableaux}
\label{sec:bst}

 In the introduction, we mentioned but did not define
irreducible characters of~$S_n$.
The only fact about irreducible characters of~$S_n$
that we need in this paper
(besides the fact that they are indexed by partitions of~$n$)
is a famous result known as the
\emph{Murnaghan--Nakayama rule}, which gives a combinatorial
rule for computing them.
In this section, we give a complete statement of the
Murnaghan--Nakayama rule,
so the reader unfamiliar with the concept of an irreducible
character may take the Murnaghan--Nakayama rule as a definition.
For more details, including a proof of the
Murnaghan--Nakayama rule, the interested reader can consult
Sagan~\cite[Section 4.10]{sagan} or Stanley \cite[Chapter 7]{stanley2}.

To state the Murnaghan--Nakayama rule, we must define
border strips (also known as rim hooks or ribbons)
and border-strip tableaux.

\begin{definition}

A \emph{border strip} is a finite set of boxes such that
in each row, the boxes in that row are contiguous,
and except for the top row, the \emph{rightmost} box in each row
lies directly underneath the \emph{leftmost} box in the row above it.
The \emph{area} or \emph{size} of a border strip is the total number of boxes.
If $B$ is a border strip, its \emph{height} $h(B)$
is the number of rows of~$B$ minus~1.
\end{definition}

In Figure~\ref{fig:borderstrip},
the four border strips have areas $11$, $7$, $1$,~$8$
and heights $3$, $2$, $0$,~$3$ respectively.
Note that the definition of a border strip
ensures that the boxes are orthogonally connected,
and that a border strip never contains two boxes $(i,j)$ and $(i',j')$
with $i<i'$ and $j<j'$.

\begin{figure}[!ht]
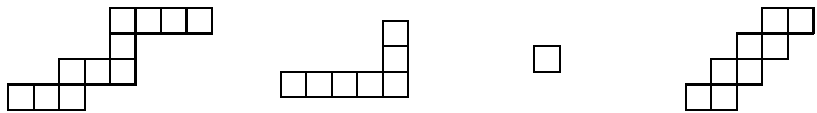

\begin{center}
\begin{ytableau}
\none & \none & \none & \none & & & & \\
\none & \none & \none & \none & \\
\none & \none &       &       & \\
      &       &        \\
\end{ytableau}
\hskip 0.3 in
\begin{ytableau}
\none & \none & \none & \none & \\
\none & \none & \none & \none & \\
\      &       &       &       &
\end{ytableau}
\hskip 0.6 in
\begin{ytableau}
\ \\
\end{ytableau}
\hskip 0.6 in
\begin{ytableau}
\none & \none & \none & & \\
\none & \none &       &   \\
\none &       & \\
      &  \\
\end{ytableau}
\end{center}
\caption{Four Examples of Border Strips}
\label{fig:borderstrip}
\end{figure}

\begin{definition}
 Let $n$ be a positive integer,
let $\lambda$ be a partition of~$n$,
and let $\alpha$ be a composition of~$n$.
A \emph{border-strip tableau (BST) of shape~$\lambda$ and type~$\alpha$}
is a tiling of the Young diagram of~$\lambda$ with
border strips such that
\begin{enumerate}
\item  the area of the $i$th border strip is~$\alpha_i$, and
\item  if the number~$i$ is written in each box of the $i$th
border strip, then the numbers weakly increase across every row and down
every column.
\end{enumerate}
\end{definition}

Figure~\ref{fig:bstexample1} shows an example of a BST
where we have colored each border strip with a different color
as a visual aid.

\begin{figure}[!ht]
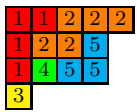

\begin{center}
\begin{ytableau}
*(red) 1 & *(red) 1 & *(orange) 2 & *(orange) 2 & *(orange) 2\\
*(red) 1 & *(orange) 2 & *(orange) 2 & *(cyan) 5 \\
*(red) 1 & *(green) 4 & *(cyan) 5 & *(cyan) 5 \\
*(yellow) 3 \\
\end{ytableau}
\end{center}
\caption{Border-Strip Tableau of Shape (5,4,4,1) and Type (4,5,1,1,3)}
\label{fig:bstexample1}
\end{figure}

 In Section~\ref{sec:dop}, we will want to consider
\emph{partial border-strip tableaux}.
We will not give a completely formal definition,
but the idea is that we take a BST
and consider only the placement of the first few border strips,
ignoring the placement of the remaining border strips.

We are now ready to state the Murnaghan--Nakayama rule.

\begin{theorem}
\label{thm:murnaghan-nakayama}
 Let $\lambda$ be a partition of~$n$,
and let $\chi_\lambda$ be the irreducible character of~$S_n$
indexed by~$\lambda$.
If $\pi\in S_n$ and $(\alpha_i)$ is
the sequence of cycle lengths of~$\pi$, then
\begin{equation}
\label{eq:murnaghan-nakayama}
\chi_\lambda(\pi) = \sum_{T} \prod_{B\in T} (-1)^{h(B)},
\end{equation}
where the sum is over all BSTs~$T$
of shape~$\lambda$ and type~$\alpha$,
and the product is over the border strips~$B$ that tile~$T$.
\end{theorem}

The expression $\prod_{B\in T} (-1)^{h(B)}$ appearing
in Equation~\eqref{eq:murnaghan-nakayama} is called
the \emph{sign} of the BST~$T$.

The alert reader may notice that
the above statement of Theorem~\ref{thm:murnaghan-nakayama}
speaks of ``the'' sequence of cycle lengths of~$\pi$,
but there is no canonical ordering
on the set of cycle lengths of a permutation.
It is a remarkable and nontrivial fact that
Theorem~\ref{thm:murnaghan-nakayama} remains true
no matter what ordering is chosen.
For example, let $\lambda = (5,4,2)$ and let $\pi$ be a permutation
with cycle lengths $6$, $3$, and~$2$.
If we let $\alpha = (6,3,2)$ then there are
two BSTs of shape~$\lambda$ and type~$\alpha$,
as shown in Figure~\ref{fig:bstexample2}.
One of these BSTs contributes $+1$ 
to the sum in Equation~\eqref{eq:murnaghan-nakayama}
and the other contributes~$-1$,
so $\chi_\lambda(\pi) = 0$.

\begin{figure}[!ht]
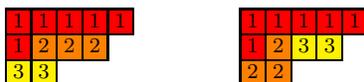

\begin{center}
\begin{ytableau}
*(red) 1 & *(red) 1  & *(red) 1  & *(red) 1  & *(red) 1  \\
*(red) 1 & *(orange) 2 & *(orange) 2 & *(orange) 2 \\
*(yellow) 3 & *(yellow) 3
\end{ytableau}
\hskip 0.5in
\begin{ytableau}
*(red) 1 & *(red) 1  & *(red) 1  & *(red) 1  & *(red) 1  \\
*(red) 1 & *(orange) 2 & *(yellow) 3 & *(yellow) 3 \\
*(orange) 2 & *(orange) 2
\end{ytableau}
\end{center}
\caption{Border-Strip Tableaux of Shape (5,4,2) and Type (6,3,2)}
\label{fig:bstexample2}
\end{figure}
On the other hand, the reader can check
that if we let $\alpha=(6,2,3)$ then there are
\emph{no} BSTs of shape~$\lambda$ and type~$\alpha$,
so again Equation~\eqref{eq:murnaghan-nakayama}
tells us that $\chi_\lambda(\pi) = 0$,
but for a ``different reason'' combinatorially.

Since the value of~$\chi_\lambda(\pi)$ depends only on
the sequence of cycle lengths of~$\pi$, for the rest of this paper
we will usually regard $\chi_\lambda$ as a function of a composition of~$n$,
and our queries will be compositions rather than permutations.

We conclude this section with a couple
of simple but useful facts about BSTs.

\begin{lemma}
\label{lem:principaldiag}
A border strip of a BST cannot contain more than one
box on the principal diagonal.
\end{lemma}

\begin{proof}
Let $i$ be the smallest positive integer
such that the border strip contains
the box of the Young diagram with coordinates $(i,i)$.
This box---call it $x$---appears in some row, say the $j$th row,
of the border strip.
By the definition of a border strip,
no box in the border strip in any row below the $j$th
can appear further to the right than $x$ does,
so in particular no other box on the principal diagonal
can appear in the border strip.
\end{proof}

\begin{lemma}
\label{lem:firstk}
In any BST, for all $k$, the first $k$ border strips
fit into the first $k$ principal hooks.
\end{lemma}

\begin{proof}
We use induction on~$k$.  The case $k = 0$ is vacuously true.
Suppose the claim is true for some $k\ge 0$.
For the claim to fail for $k+1$,
the $(k+1)$st border strip must contain a box~$x$
in the $i$th principal hook for some $i \ge k+2$.
Then since the box $(k+2,k+2)$
lies (weakly) above and to the left of~$x$,
the weakly increasing property of a BST
forces $(k+2,k+2)$ to belong to border strip~$j$ for some $j\le k+1$.
But by the induction hypothesis, $(k+2,k+2)$
does not belong to any of the first $k$ border strips, so $j=k+1$.
The same argument proves that the box $(k+1,k+1)$
must also belong to border strip $k+1$.
This contradicts Lemma~\ref{lem:principaldiag}.
\end{proof}

\section{High-Level Structure of the Algorithm}
\label{sec:sketch}

The algorithm consists of a \emph{forward pass} followed by
a \emph{backward pass}.
As we explain in Section~\ref{sec:forward},
during the forward pass, we determine
the principal hook lengths~$h_i$ one at a time, in the order
$h_1, h_2, h_3, \ldots\,$.
As we explain in Section~\ref{sec:backward},
during the backward pass, we determine the principal hooks
(i.e., their actual shapes, not just their lengths) in \emph{reverse} order,
starting with the last principal hook and working backward.
It turns out that for the backward pass,
it suffices to show how to recover
the 1st principal hook given that we know all the later hooks,
and we accomplish this in two steps:
\begin{enumerate}
\item  Determine the overhang lengths $a_1$ and~$b_1$.
We give the details in Section~\ref{sec:overhang}.
This reduces the problem to distinguishing between $\lambda$
and its \dop~$\hat\lambda$.
\item  Distinguish between $\lambda$ and~$\hat\lambda$.
This step is surprisingly complicated, and subdivides into
several cases, which we explain in detail in Section~\ref{sec:dop}.
\end{enumerate}

\subsection{The Forward Pass}
\label{sec:forward}

In the forward pass, we determine the
principal hook lengths of~$\lambda$.
To determine~$h_1$, we define, for each $i\in \{0,1,\ldots,n-1\}$,
the composition~$\alpha^{(i)}$ of~$n$ by
\begin{equation*}
\alpha^{(i)}_j =
\begin{cases}
n - i, &\mbox{if $j=1$;}\\
1, &\mbox{if $2 \le j \le i+1$.}
\end{cases}
\end{equation*}
We submit the queries
$\alpha^{(0)}, \alpha^{(1)}, \alpha^{(2)}, \ldots$
successively to the oracle,
stopping as soon as we encounter a nonzero value.

The key observation is that by Lemma~\ref{lem:firstk} (with $k=1$),
there cannot exist a BST
of shape~$\lambda$ and type~$\alpha^{(i)}$
if $\alpha^{(i)}_1 > h_1$,
because the 1st border strip would be simply too large
to fit inside the 1st principal hook.
So if $\alpha^{(i)}_1 > h_1$,
then $\chi_\lambda(\alpha^{(i)}) = 0$.
Conversely, if $\alpha^{(i)}_1 = h_1$,
then the BSTs of shape~$\lambda$ and type~$\alpha^{(i)}$
are obtained by letting the 1st border strip
cover the entire 1st principal hook,
and then arranging the numbers from $2$ through $i+1$
in the rest of~$\lambda$ so that they increase across rows and down columns.
There is always at least one such BST;
if there are more, then they all have the same sign,
and hence in particular, $\chi_\lambda(\alpha^{(i)}) \ne 0$.
Therefore, if $i_0$ is the smallest~$i$ such that
$\chi_\lambda(\alpha^{(i)}) \ne 0$,
then $h_1 = \alpha^{(i_0)}_1 =  n - i_0$.

Once we know $h_1$, we can determine $h_2$ by a similar procedure.
If $h_1 = n$ then we are done.
Otherwise, we fix the size of the 1st border strip at~$h_1$,
and find, by repeated ``guessing and checking,''
the largest possible size of
the 2nd border strip; this is~$h_2$.  More formally,
we define compositions~$\beta^{(i)}$ of~$n$ as follows:
\begin{equation*}
\beta^{(i)}_j = 
\begin{cases}
h_1, & \mbox{if $j=1$;}\\
\min(h_1 - 2, n - h_1) - i, & \mbox{if $j=2$;}\\
1, &\mbox{if $3\le j \le n - \beta^{(i)}_1 - \beta^{(i)}_2 + 2$.}
\end{cases}
\end{equation*}
We submit the queries
$\beta^{(0)}, \beta^{(1)}, \beta^{(2)}, \ldots$
successively to the oracle,
stopping as soon as we encounter a nonzero value.
Since $\beta^{(i)}_1 = h_1$, Lemma~\ref{lem:firstk} tells us
that in any BST of shape~$\lambda$ and type~$\beta^{(i)}$,
the 1st border strip must entirely cover the 1st principal hook.
If $\beta^{(i)}_2 > h_2$, then there cannot be any BSTs
of shape~$\lambda$ and type~$\beta^{(i)}$,
because the 2nd border strip would be too large to fit inside
the 2nd principal hook.
Conversely, if $\beta^{(i)}_2 = h_2$,
then the BSTs of shape~$\lambda$ and type~$\beta^{(i)}$
are obtained by letting the 1st and 2nd border strips
cover the entire 1st and 2nd principal hooks respectively,
and then arranging the numbers from
$3$ to $n - \beta^{(i)}_1 - \beta^{(i)}_2 + 2$
in the rest of~$\lambda$ so that they increase
across rows and down columns.
As before, there is always at least one BST
of this form, and if there are more, then they all have the same sign.
So if $i_0$ is the smallest~$i$ such that
$\chi_\lambda(\beta^{(i)}) \ne 0$,
then $h_2 = \beta^{(i_0)}_2$.

The pattern should now be clear.
Given that we know the first few principal hook lengths,
we fix the corresponding border-strip sizes to be equal to
the known principal hook lengths,
and use the next border strip to ``guess''
the size of the next principal hook length,
starting with the largest conceivable value
and working our way downward.
If our guess is too large, then the oracle will return zero.
As soon as the oracle returns a nonzero value for $\chi_\lambda$,
that tells us that our guess for
the size of the next principal hook length is correct.
In this way, we can recover all the principal hook lengths.

\subsection{The Backward Pass}
\label{sec:backward}

Knowing the principal hook lengths of~$\lambda$ does not, in general,
determine~$\lambda$ uniquely, because there can be many
different hooks with the same hook length.
Our overall strategy for determining~$\lambda$
will be to recover the principal hooks themselves (not just their lengths)
inductively, starting with the \emph{last} (or innermost) principal hook,
and working backward one principal hook at a time.
Each principal hook will be recovered in two steps;
first, we will recover the overhang lengths $a_i$ and~$b_i$,
and then (if $a_i \ne b_i$) we will recover which of $a_i$ and~$b_i$ is
the arm overhang length and which is the leg overhang length.

In this backward pass,
we claim that we may assume without loss of generality that
we know $\lambda$ completely except that we are unsure about
the shape of the 1st principal hook.
To see this, suppose that we know only the shapes of
the $j$th principal hooks of~$\lambda$
for $j$ greater than some value~$j_0$,
and we want to recover the $j_0$th principal hook of~$\lambda$.
Consider what happens if $\alpha_i = h_i$ for all $i<j_0$.
If we query the value of $\chi_\lambda(\alpha)$,
then by the same kind of argument we gave in
Section~\ref{sec:forward},
each of the first $j_0-1$ border strips is forced to cover
the corresponding principal hook entirely.
In effect, we are querying a smaller shape $\mu$---one that has been
obtained from~$\lambda$ by ``stripping off'' the first $j_0-1$ principal hooks,
and that we know completely except for its 1st principal hook.
In other words, given a sequence of queries for~$\mu$
that recover its 1st principal hook, we can simply prepend
$\alpha_i = h_i$ for all $i<j_0$ to each of these queries;
this will give us a sequence of queries for~$\lambda$
that allow us  to recover its $j_0$th principal hook.

Our description of the backward pass,
and our proof of its correctness,
proceeds by induction on the number of principal
hooks of~$\lambda$.
The induction step---which assumes that
the exact shapes of all the principal hooks~$H_i$
with $i\ge 2$ are known,
and shows how to recover the principal hook~$H_1$---is
the topic of Section~\ref{sec:firsthook}.

Let us now establish the base case,
when we know that $\lambda$ is a hook, and we know its length~$h_1$,
but we do not know its precise shape.
Let $\lambda_1$ denote the number of boxes in the first row of~$\lambda$.
If we let $\alpha_i=1$ for $1\le i\le n$, then
every BST of shape~$\lambda$ and type~$\alpha$ has positive sign,
and Theorem~\ref{thm:murnaghan-nakayama} implies that
\begin{equation*}
\chi_\lambda(\alpha) = \binom{n-1}{\lambda_1-1},
\end{equation*}
because we may choose any $\lambda_1 - 1$ of the numbers
in the set $\{2,3,\ldots, n\}$ to place in the first row
of our BST, and everything else about the BST
is uniquely determined.
For fixed $n$, the binomial coefficients $\binom{n-1}{r}$ are distinct,
except that $\binom{n-1}{r} = \binom{n-1}{n-1-r}$.
Therefore once we know $\chi_\lambda(\alpha)$,
we know that our shape is either $\lambda$ or its \dop~$\hat\lambda$.
If $\lambda=\hat\lambda$ (i.e., $\lambda$ is self-conjugate)
then we are done.
Otherwise, we define
\begin{equation*}
\beta_i :=
\begin{cases}
2, &\mbox{if $i=1$;}\\
1, &\mbox{if $2\le i \le n-1$.}
\end{cases}
\end{equation*}
We query the value of $\chi_\lambda(\beta)$.
There are two types of BSTs of shape~$\lambda$
and type~$\beta$, depending on whether the first border strip
is arranged horizontally or vertically;
in the former case, the first border strip has height~$0$
so the BST has positive sign,
whereas in the latter case, the first border strip has height~$1$
so the BST has negative sign.
A straightforward application of
Theorem~\ref{thm:murnaghan-nakayama} yields
\begin{equation*}
\chi_\lambda(\beta) = \binom{n-2}{\lambda_1-2} - \binom{n-2}{\lambda_1 - 1}.
\end{equation*}
On the other hand, $\chi_{\hat\lambda}(\beta) = - \chi_{\lambda}(\beta)$.
Therefore, as long as $\chi_\lambda(\beta)\ne 0$,
we can distinguish between $\lambda$ and~$\hat\lambda$
just by examining whether the answer to our query is positive or negative.
But the only way that $\chi_\lambda(\beta)$ can be zero is if
$(\lambda_1 - 2) + (\lambda_1 - 1) = n-2$, which can happen only if $\lambda$
is self-conjugate---a case that we already dealt with above.

\section{Recovering the First Principal Hook}
\label{sec:firsthook}

For the remainder of this paper, we make the 
following standing assumptions: $\lambda$ is a partition
with principal hook lengths $h_i$
and overhang lengths $a_i$ and~$b_i$, where $a_i\le b_i$.
We assume that $\lambda$ has exactly $k+1$ nonempty principal hooks,
with $k\ge 1$.
We further assume that we know all the~$h_i$
as well as the exact shape formed by
all the principal hooks~$H_i$ with $i\ge 2$.

\subsection{Determining the Overhang Lengths
\texorpdfstring{$\bm{a_1}$}{a1}
and \texorpdfstring{$\bm{b_1}$}{b1}}
\label{sec:overhang}

Our goal in this section is to explain how to determine
the overhang lengths $a_1$ and~$b_1$.
We define compositions~$\alpha^{(i)}$ of~$n$ as follows:
\begin{equation*}
\alpha_j^{(i)} = 
\begin{cases}
h_1 - i, &\mbox{if $j=1$;} \\
h_2 + i, &\mbox{if $j=2$;} \\
h_j,     &\mbox{if $j\ge3$.}
\end{cases}
\end{equation*}
We submit the queries
$\alpha^{(1)}, \alpha^{(2)}, \alpha^{(3)}, \ldots$
successively to the oracle,
stopping as soon as we encounter a nonzero value.

Suppose that for some $i\ge 1$, there exists a BST~$T$
of shape~$\lambda$ and type~$\alpha^{(i)}$.
Lemma~\ref{lem:firstk} implies that the 2nd border strip
must be contained within the first two principal hooks,
and because $\alpha^{(i)}_2 > h_2$,
the 2nd border strip must contain some---in fact,
exactly~$i$---boxes from the 1st principal hook.
Let $x$ be the box in the 1st principal hook
adjacent to the overhang of length~$a_1$,
and let $y$ be the box in the 1st principal hook
adjacent to the overhang of length~$b_1$.
Then the 1st border strip cannot cover \emph{both}
$x$ and~$y$, because then the 2nd border strip
(which must in particular contain the box $(2,2)$)
would not be able to contain any boxes from the 1st principal hook.
See for example Figure~\ref{fig:toolarge},
where the 1st border strip
(which we have colored red,
omitting the 1s in the boxes to avoid clutter)
contains both $x$ and~$y$, and thus blocks
the 2nd border strip (in orange) from containing any boxes from
the 1st principal hook.

\begin{figure}[!ht]
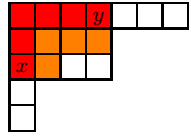

\begin{center}
\begin{ytableau}
*(red) & *(red) & *(red) & *(red) y & & & \\
*(red) & *(orange) & *(orange) & *(orange) \\
*(red) x & *(orange) &  & \\
\  \\
\  \\
\end{ytableau}
\end{center}
\caption{1st Border Strip Contains Both $x$ and $y$}
\label{fig:toolarge}
\end{figure}
But if $i\le a_1$ then the 1st border strip
is so long that it is forced to cover both $x$ and~$y$.
Hence if $i\le a_1$, then $\chi_\lambda(\alpha^{(i)}) = 0$.

Conversely, suppose that $i = a_1+1$.
Then there is a BST of shape~$\lambda$
and type~$\alpha^{(i)}$ in which
\begin{itemize}
\item 
the 1st border strip covers everything in the 1st principal hook
except $x$ and the overhang adjacent to~$x$;
\item 
the 2nd border strip covers $x$, the overhang adjacent to~$x$,
and the entire 2nd principal hook; and
\item 
for $j>2$, the $j$th border strip covers the entire $j$th principal hook.
\end{itemize}
\noindent See for example Figure~\ref{fig:alphaa1}.

\begin{figure}[!ht]
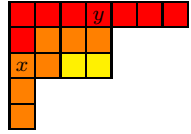

\begin{center}
\begin{ytableau}
*(red) & *(red) & *(red) & *(red) y & *(red) & *(red) & *(red) \\
*(red) & *(orange) & *(orange) & *(orange) \\
*(orange) x & *(orange) & *(yellow) & *(yellow) \\
*(orange)  \\
*(orange)  \\
\end{ytableau}
\end{center}
\caption{Border-Strip Tableau of Type $\alpha^{(a_1+1)}$}
\label{fig:alphaa1}
\end{figure}
Moreover, if $a_1 < b_1$, then this is the \emph{only}
BST of shape~$\lambda$ and type~$\alpha^{(a_1+1)}$,
because any other placement of the 1st border strip
would cover both $x$ and~$y$,
and once the 1st border strip is placed,
the areas of the border strips
(together with the constraint imposed by Lemma~\ref{lem:firstk})
force a unique placement of the remaining border strips.
Therefore, $\chi_\lambda(\alpha^{(a_1+1)}) = \pm 1$,
and $i = a_1+1$ is the smallest value of~$i$
such that $\chi_\lambda(\alpha^{(i)}) \ne 0$.

On the other hand, if $a_1 = b_1$, then there is
a second BST of shape~$\lambda$ and type~$\alpha^{(a_1+1)}$,
where the 1st border strip covers everything in the 1st
principal hook except~$y$ and the overhang adjacent to~$y$.
In the case of one of these two BSTs, the heights of the first two
border strips are $\lambda_1'-1$ and $\lambda_1'-a_1-1$,
and in the other case, the heights of the first two border strips
are $\lambda_1'-a_1-2$ and $\lambda_1'-2$.
Therefore the two BSTs have the same sign, so
$\chi_\lambda(\alpha^{(a_1+1)}) = \pm 2$,
and again $i = a_1+1$ is the smallest value of~$i$
such that $\chi_\lambda(\alpha^{(i)}) \ne 0$.

Once the value of $a_1$ is determined, the value of~$b_1$ is also
determined, since $h_1 = h_2 + a_1 + b_1 + 2$.

\subsection{Distinguishing \Dop s}
\label{sec:dop}

Our goal in this section is to
construct a query~$\alpha$ that distinguishes
between $\lambda$ and its \dop; i.e., such that
$\chi_\lambda(\alpha) \ne \chi_{\hat\lambda}(\alpha)$.
We assume that $a_1 \ne b_1$,
since otherwise $\lambda = \hat\lambda$.

\begin{definition}

The \emph{second imbalance} of~$\lambda$,
denoted $I(\lambda)$, is the smallest integer $i\ge 2$
such that $a_i \ne b_i$.
If $a_i = b_i$ for $2\le i \le k+1$,
then we set $I(\lambda) := \infty$.
\end{definition}

The reason for calling $I(\lambda)$ the ``second''
imbalance is that the first imbalance is always~$1$,
since we have assumed that $a_1 \ne b_1$.
Figure~\ref{fig:imbalance5} illustrates
a partition with $k=4$ and $I(\lambda) = 5$.

\begin{figure}[!ht]
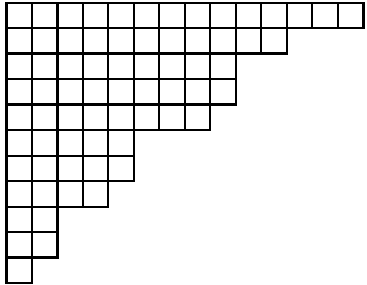

\begin{center}
\begin{ytableau}
\ & & & & & & & & & & & & & \\
\ & & & & & & & & & &       \\
\ & & & & & & & &           \\
\ & & & & & & & &           \\
\ & & & & & & &             \\
\ & & & &                   \\
\ & & & &                   \\
\ & & &                     \\
\ &                         \\
\ &                         \\
\ 
\end{ytableau}
\end{center}
\caption{Partition with $k=4$ and $I(\lambda) = 5$}
\label{fig:imbalance5}
\end{figure}

Throughout this section,
our queries will always start off the same way,
as described in the following crucial definition.

\begin{definition}

Let $m := \min(I(\lambda),k+1) - 1$,
or in other words let $m=k$ if $I(\lambda)=\infty$
and $m=I(\lambda)-1$ if $I(\lambda)<\infty$.  Define
\begin{equation}
\label{eq:alphalambda}
\alpha_i(\lambda) =
\begin{cases}
h_1 - a_1 - 1,         &\mbox{if $i=1$;} \\
h_i - a_i + a_{i-1},   &\mbox{if $2 \le i \le m$.}
\end{cases}
\end{equation}
\end{definition}

\subsubsection{The Greedy Arrangement}

 As we now explain, there is always a way to create
a partial border-strip tableau of shape~$\lambda$
in which the $i$th border strip has area $\alpha_i(\lambda)$,
for $1 \le i \le m$.
Namely, let us place the 1st border strip so that it
entirely covers the 1st long overhang.
The area of the 1st border strip is
$\alpha_1(\lambda) = h_1 - a_1 - 1$,
which is less than $h_1$, so it is short enough to
fit inside the 1st principal hook.
We must also check that it is long enough to reach
the box in the $(1,1)$ position,
since in any BST,
the 1st border strip must cover $(1,1)$.
Because its area is $h_1 - a_1 - 1$,
what the 1st border strip leaves uncovered (inside the 1st principal hook)
is the 1st short overhang and
the box adjacent to the 1st short overhang---call this box~$x_1$.
Since $\lambda$ has at least~2 nonempty principal hooks,
$x_1$ is \emph{not} in the $(1,1)$ position,
so the 1st border strip does indeed cover $(1,1)$.
See the red border strip in Figure~\ref{fig:greedy} for an example.

\begin{figure}[!ht]
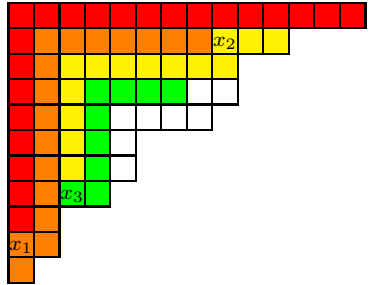

\begin{center}
\begin{ytableau}
*(red) & *(red) & *(red) & *(red) & *(red) & *(red) & *(red) & *(red)
   & *(red) & *(red) & *(red) & *(red) & *(red) & *(red) \\
*(red) & *(orange) & *(orange) & *(orange) & *(orange) & *(orange)
   & *(orange) & *(orange) & *(yellow) x_2 & *(yellow) & *(yellow) \\
*(red) & *(orange) & *(yellow) & *(yellow) & *(yellow) & *(yellow)
   & *(yellow) & *(yellow) & *(yellow)    \\
*(red) & *(orange) & *(yellow) & *(green) & *(green) & *(green) & *(green) & & \\
*(red) & *(orange) & *(yellow) & *(green) & & & &             \\
*(red) & *(orange) & *(yellow) & *(green) &                   \\
*(red) & *(orange) & *(yellow) & *(green) & \\
*(red) & *(orange) & *(green) x_3 & *(green)  \\
*(red) & *(orange)    \\
*(orange) x_1 & *(orange) \\
*(orange) 
\end{ytableau}
\end{center}
\caption{Greedy Arrangement}
\label{fig:greedy}
\end{figure}

 If $m \ge 2$, then let us place the 2nd border strip
so that it entirely covers the 1st short overhang.  Because $x_1$ is
\emph{not} in the 1st short overhang, it is adjacent to some box in
the 2nd principal hook; indeed, it is adjacent to one end of the 2nd
principal hook.  Thus the 2nd border strip can ``spill over'' into the
2nd principal hook; indeed, it spills over by precisely
\begin{equation*}
\alpha_2(\lambda) - (a_1 + 1) = h_2 - a_2 + a_1 - (a_1 + 1) = h_2 - a_2 - 1
\end{equation*}
boxes.
Now since $m \ge 2$,
the 2nd arm overhang and the 2nd leg overhang both have length~$a_2$,
and there are at least~3 nonempty principal hooks.
Arguing as we did for the 1st border strip,
we conclude that what the 2nd border strip leaves uncovered
(in the 2nd principal hook)
is the overhang at the opposite end of the 2nd principal hook from~$x_1$,
plus the box (call it~$x_2$) adjacent to that overhang.
Moreover, positioning the 2nd border strip in this fashion
does indeed cover the $(2,2)$ box, as it is required to do.
See the orange border strip in Figure~\ref{fig:greedy}.

The placements of the remaining border strips
(if they exist) follow the same recipe.  The $i$th border
strip is positioned to cover the boxes in the $(i-1)$st
principal hook that are \emph{not} covered by the
$(i-1)$st border strip, namely the uncovered overhang
and the box~$x_{i-1}$ adjacent to that overhang.
The border strip then ``spills over'' into the $i$th
principal hook, covering everything except the overhang
at the opposite end of the $i$th principal hook and
the box~$x_i$ adjacent to that overhang.  This recipe works
even when an overhang is empty, as illustrated by
the green border strip in Figure~\ref{fig:greedy}.
We formalize this construction with the following definition.

\begin{definition}
Let $(\alpha_i(\lambda))_{i=1,\ldots,m}$ be
given by Equation~\eqref{eq:alphalambda}.
A BST of shape~$\lambda$
is called \emph{greedy} if the first $m$ border strips
have areas $(\alpha_i(\lambda))_{i=1,\ldots,m}$ and
are arranged in the manner described in the preceding paragraphs.
It is called \emph{non-greedy} if the first $m$ border strips
have areas $(\alpha_i(\lambda))_{i=1,\ldots,m}$ but
some of those first $m$ border strips are arranged differently.
\end{definition}

In Sections \ref{subsec:nongreedy} to \ref{subsec:Ik1},
we deal with the case when $I(\lambda) < \infty$.
In a nutshell, our algorithm will submit queries~$\alpha$
whose first few values are given by Equation~\eqref{eq:alphalambda},
and we will argue that, for the values of~$\alpha$ that we choose,
\begin{itemize}
\item 
the number of non-greedy BSTs of shape~$\lambda$
and type~$\alpha$
equals the number of non-greedy BSTs of shape~$\hat\lambda$
and type~$\alpha$, and
\item 
the number of greedy BSTs of shape~$\lambda$ and type~$\alpha$
has the opposite parity from 
the number of greedy BSTs of shape~$\hat\lambda$ and type~$\alpha$.
\end{itemize}
Given these two facts, it follows that,
regardless of the signs of the BSTs,
$\chi_\lambda(\alpha) \not\equiv \chi_{\hat\lambda}(\alpha) \pmod 2$,
and therefore
$\chi_\lambda(\alpha) \ne \chi_{\hat\lambda}(\alpha)$.

Finally, in Section \ref{subsec:Iinfty}, we deal with
the case $I(\lambda)=\infty$.

\subsubsection{Non-Greedy Border-Strip Tableaux}
\label{subsec:nongreedy}

In what follows, we will overload the notation~$\alpha_i$,
using it to mean both the $i$th part of a composition~$\alpha$,
as well as the function given by Equation~\eqref{eq:alphalambda}.
There should be no confusion because we will only be considering
compositions~$\alpha$ whose first few parts coincide with the
values given in Equation~\eqref{eq:alphalambda}.

Let us make some general observations about non-greedy BSTs.
Let $\alpha$ be a composition of~$n$ with
$\alpha_i = \alpha_i(\lambda)$ for $1\le i \le m$.
Let $T$ be a non-greedy BST of shape~$\lambda$
and type~$\alpha$.
Let $i\le m$ be the smallest number such that
the $i$th border strip is \emph{not} positioned as the
greedy arrangement dictates.
We claim that the only other allowable positions of
the $i$th border strip are \emph{slides} of its greedy position,
meaning that we remove some boxes from one end of the border strip
and append the same number of boxes to the other end;
if we move $j\ge 1$ boxes from one end to the other then
we call the resulting border strip the \emph{$j$th slide}.
For example, if $\lambda=(14,11,9,9,8,5,4,4,2,2,1)$,
then Figure~\ref{fig:nongreedy} illustrates
the two possible slides of the 1st border strip if $i=1$.

\begin{figure}[!ht]
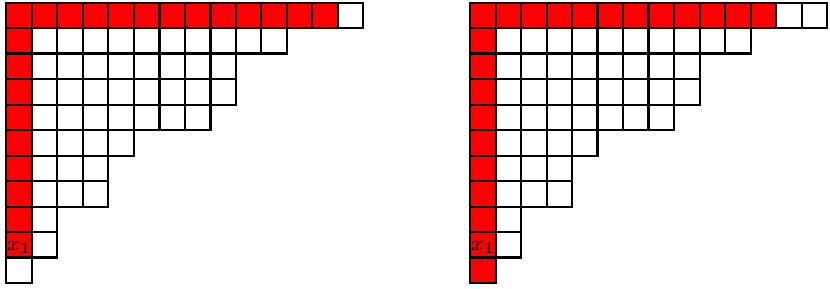

\begin{center}
\begin{ytableau}
*(red) & *(red) & *(red) & *(red) & *(red) & *(red) & *(red) & *(red)
   & *(red) & *(red) & *(red) & *(red) & *(red) & \\
*(red) &  &  &  &  & 
   &  &  &  &  &  \\
*(red) &  &  &  &  & 
   &  &  &     \\
*(red) &  &  &  &  &  &  & & \\
*(red) &  &  &  & & & &             \\
*(red) &  &  &  &                   \\
*(red) &  &  &  \\
*(red) &  &  &   \\
*(red) &     \\
*(red) x_1 &  \\
\ 
\end{ytableau}
\hskip 0.5in
\begin{ytableau}
*(red) & *(red) & *(red) & *(red) & *(red) & *(red) & *(red) & *(red)
   & *(red) & *(red) & *(red) & *(red) & & \\
*(red) &  &  &  &  & 
   &  &  &  &  &  \\
*(red) &  &  &  &  & 
   &  &  &     \\
*(red) &  &  &  &  &  &  & & \\
*(red) &  &  &  & & & &             \\
*(red) &  &  &  &                   \\
*(red) &  &  &  \\
*(red) &  &  &   \\
*(red) &     \\
*(red) x_1 &  \\
*(red)
\end{ytableau}
\end{center}
\caption{Non-Greedy Arrangements of 1st Border Strip}
\label{fig:nongreedy}
\end{figure}

 Figure~\ref{fig:nongreedy2} illustrates, for the same~$\lambda$,
the 1st and 3rd slides of the 2nd border strip if $i=2$.

\begin{figure}[!ht]
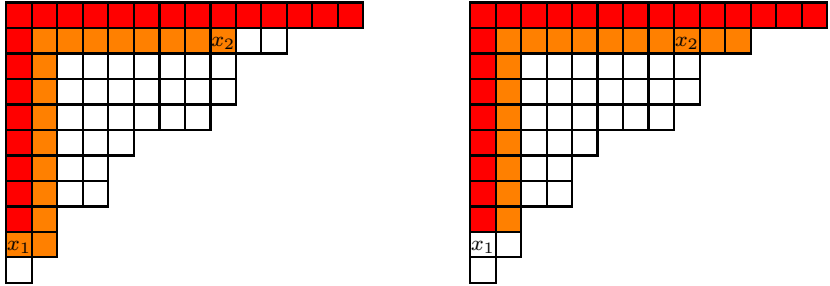

\begin{center}
\begin{ytableau}
*(red) & *(red) & *(red) & *(red) & *(red) & *(red) & *(red) & *(red)
   & *(red) & *(red) & *(red) & *(red) & *(red) & *(red) \\
*(red) & *(orange) & *(orange)  & *(orange)  & *(orange)  & *(orange) 
   &  *(orange) & *(orange)  & *(orange) x_2 &  &  \\
*(red) & *(orange)   &  &  &  & 
   &  &  &     \\
*(red) & *(orange)   &  &  &  &  &  & & \\
*(red) & *(orange)   &  &  & & & &             \\
*(red) & *(orange)   &  &  &                   \\
*(red) & *(orange)   &  &  \\
*(red) & *(orange)   &  &   \\
*(red) & *(orange)  \\
*(orange) x_1 & *(orange)  \\
\ 
\end{ytableau}
\hskip 0.5in
\begin{ytableau}
*(red) & *(red) & *(red) & *(red) & *(red) & *(red) & *(red) & *(red)
   & *(red) & *(red) & *(red) & *(red) & *(red) & *(red) \\
*(red) & *(orange) & *(orange)  & *(orange)  & *(orange)  & *(orange) 
   &  *(orange) & *(orange)  & *(orange) x_2 & *(orange) & *(orange) \\
*(red) & *(orange)   &  &  &  & 
   &  &  &     \\
*(red) & *(orange)   &  &  &  &  &  & & \\
*(red) & *(orange)   &  &  & & & &             \\
*(red) & *(orange)   &  &  &                   \\
*(red) & *(orange)   &  &  \\
*(red) & *(orange)   &  &   \\
*(red) & *(orange)  \\
x_1 & \\
\ 
\end{ytableau}
\end{center}
\caption{1st and 3rd Slides
of 2nd Border Strip}
\label{fig:nongreedy2}
\end{figure}

To see why slides are the only allowable positions
of the $i$th border strip, recall that Lemma~\ref{lem:firstk}
implies that the first $i$ border strips must lie within the first
$i$ principal hooks.
The only available boxes for
the $i$th border strip are in the $(i-1)$st and
$i$th principal hooks,
because the greedy arrangement entirely fills all the
earlier principal hooks.

 We now claim that the only boxes in the $i$th principal hook
of~$T$ that are not covered by the $i$th border strip
must lie in the $i$th arm overhang or the $i$th leg overhang.

To see why the claim is true if $i=1$,
note that every slide of the 1st border strip
necessarily covers the box~$x_1$,
and any uncovered boxes in the 1st principal hook
lying beyond~$x_1$ must be in the 1st short overhang.
On the other hand, 
even if the 1st border strip is slid as far as possible,
so that it covers the entire 1st short overhang,
it will leave uncovered at most $a_1 + 1 \le b_1$ boxes,
which must therefore all lie in the 1st long overhang.

A similar argument establishes the claim if $i>1$.
For example, if $i=2$, then every slide of the 2nd border strip
covers~$x_2$, and any uncovered boxes in the 2nd principal hook
lying beyond~$x_2$ must be in an overhang.
On the other hand, even if the 2nd border strip is slid as far
as possible, it will leave uncovered
\begin{equation*}
h_2 - (h_2 - a_2 + a_1) = a_2 - a_1 \le a_2
\end{equation*}
boxes at one end of the 2nd principal hook,
so these boxes must all lie in an overhang.

In fact, we can say a little more.
Consider the set~$S$ of boxes in the first $i$ principal hooks
that are \emph{not} covered by the first $i$ border strips.
There are $j$ boxes in~$S$ that were covered in the greedy
arrangement and that are now uncovered as a result of the slide;
call this set of boxes the \emph{outer island}
(see for example the blue boxes in Figure~\ref{fig:bijection}).
The outer island comprises some boxes in an $(i-1)$st overhang,
plus possibly $x_{i-1}$ and some boxes in the adjacent $i$th overhang.
The remaining boxes in~$S$ lie in the other $i$th overhang,
at the other end of the $i$th border strip;
call this set of boxes the \emph{inner island}
(see for example the violet boxes in Figure~\ref{fig:bijection};
note that the inner island may be empty).
The reason we call these sets ``islands''
is that they are disconnected from each other
and from all later principal hooks;
i.e., no border strip can contain both a box from an
(inner or outer) island and a box from the $i'$-th principal
hook for $i'>i$,
because they are disconnected from each other by
the $i$th border strip.
Finally, note also that neither island can contain
a box from the principal diagonal;
boxes in an overhang can never lie on the principal diagonal,
so the only worry is that $x_{i-1}$ might lie on the principal diagonal,
but this can happen only if $x_{i-1}$ is on the innermost hook,
which is not possible because $x_{i-1}$ is on the
$(i-1)$st principal hook
and there are at least $i$ principal hooks.

\begin{lemma}
\label{lem:nongreedy1}

If $\alpha$ is a composition of~$n$ with
$\alpha_i = \alpha_i(\lambda)$ for $1\le i \le m$, then the number of
non-greedy BSTs of shape~$\lambda$ and type~$\alpha$
is equal to the number of
non-greedy BSTs of shape~$\hat\lambda$ and type~$\alpha$.
\end{lemma}

\begin{proof}

We describe a bijection between
non-greedy BSTs of shape~$\lambda$ and type~$\alpha$ and
non-greedy BSTs of shape~$\hat\lambda$ and type~$\alpha$.
As above, if $T$ is a non-greedy BST of shape~$\lambda$
and type~$\alpha$, we let $i$ be the smallest number such that
the $i$th border strip is \emph{not} positioned greedily;
the $i$th border strip in~$T$ is then the $j$th slide for some~$j$.
To construct the corresponding BST~$\hat T$
of shape~$\hat \lambda$, we proceed as follows:
\begin{enumerate}
\item \label{item:bijection1} 
Arrange the first $i-1$ border strips of~$\hat T$ greedily.
\item \label{item:bijection2} 
Put the $i$th border strip of $\hat T$ in the $j$th slide position.
\item \label{item:bijection3} 
For each remaining border strip, if, in~$T$,
it lies in the inner or outer island, then in~$\hat T$,
put it in the \emph{transposed} location (explained below).
Otherwise, put it in the same place that it appears in~$T$.
\end{enumerate}
Step~\ref{item:bijection3} requires further explanation.
Suppose without loss of generality that the 1st long overhang
of~$\lambda$ is its 1st arm overhang (as opposed to the 1st leg overhang),
so that the 1st long overhang of~$\hat\lambda$ is its 1st leg overhang.
Then in the greedy arrangement of the first $i-1$ border strips of~$T$,
an extreme end of the 1st, 2nd, 3rd, 4th, $\ldots$
border strips will be at the end of an arm, leg, arm, leg, $\ldots$
respectively, whereas
in the greedy arrangement of the first $i-1$ border strips of~$\hat T$,
an extreme end of the 1st, 2nd, 3rd, 4th, $\ldots$
border strips will be at the end of a leg, arm, leg, arm, $\ldots\,$.
Therefore if the $i$th border strip of~$T$ is slid by~$j$ from the
end of an arm, then the $i$th border strip of~$\hat T$ will be slid
by~$j$ from the end of a leg, and vice versa.
The inner and outer islands in $T$ and~$\hat T$
will therefore be identical except that
they will be transposed---the $(i-1)$st and $i$th rows
in~$T$ will become the $(i-1)$st and $i$th columns in~$\hat T$,
and vice versa.  See Figure~\ref{fig:bijection} for an example with $i=2$,
where the outer and inner islands
are colored blue and violet respectively.

\begin{figure}[!ht]
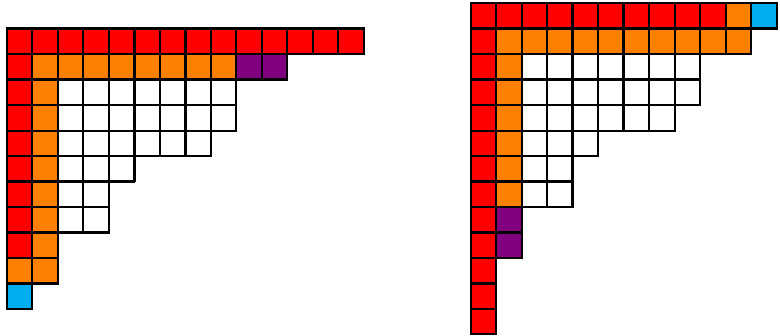

\begin{center}
\begin{ytableau}
*(red) & *(red) & *(red) & *(red) & *(red) & *(red) & *(red) & *(red)
   & *(red) & *(red) & *(red) & *(red) & *(red) & *(red) \\
*(red) & *(orange) & *(orange)  & *(orange)  & *(orange)  & *(orange) 
   &  *(orange) & *(orange)  & *(orange) & *(violet) & *(violet) \\
*(red) & *(orange)   &  &  &  & 
   &  &  &     \\
*(red) & *(orange)   &  &  &  &  &  & & \\
*(red) & *(orange)   &  &  & & & &             \\
*(red) & *(orange)   &  &  &                   \\
*(red) & *(orange)   &  &  \\
*(red) & *(orange)   &  &   \\
*(red) & *(orange)  \\
*(orange) & *(orange)  \\
*(cyan)
\end{ytableau}
\hskip 0.5in
\begin{ytableau}
*(red) & *(red) & *(red) & *(red) & *(red) & *(red) & *(red) & *(red)
   & *(red) & *(red) & *(orange) & *(cyan) \\
*(red) & *(orange) & *(orange)  & *(orange)  & *(orange)  & *(orange) 
   &  *(orange) & *(orange)  & *(orange) & *(orange) & *(orange) \\
*(red) & *(orange)   &  &  &  & 
   &  &  &     \\
*(red) & *(orange)   &  &  &  &  &  & & \\
*(red) & *(orange)   &  &  & & & &             \\
*(red) & *(orange)   &  &  &                   \\
*(red) & *(orange)   &  &  \\
*(red) & *(orange)   &  &   \\
*(red) & *(violet) \\
*(red) & *(violet) \\
*(red) \\
*(red) \\
*(red)
\end{ytableau}
\end{center}
\caption{$T$ and $\hat T$ with $i=2$ and $j=1$}
\label{fig:bijection}
\end{figure}
This is what we mean in Step~\ref{item:bijection3} above by ``transposed'';
the border strips that ultimately cover the inner and outer islands in~$T$
should be used to cover the inner and outer islands
in the transposed positions in~$\hat T$.
This procedure will never cause a violation of the
weakly-increasing property of a BST,
because the inner and outer islands
will be covered by border strips beyond the~$i$th.

Verifying that our putative bijection is indeed a bijection
is straightforward and we will not belabor the details.
\end{proof}

\subsubsection{The Case
\texorpdfstring{$\bm{I(\lambda)\le k}$}{Ik}} 
\label{subsec:Ilek}

\begin{lemma}
\label{lem:case0}

Assume that $I := I(\lambda) \le k$.
For $1 \le i < I$, let $\alpha_i = \alpha_i(\lambda)$
as defined by Equation~\eqref{eq:alphalambda}. Let
\begin{equation*}
\alpha_i :=
\begin{cases}
h_i - a_i +a_{i-1}, &\mbox{\rm {if} $i = I$;} \\
h_i + a_{i-1} + 1, &\mbox{\rm {if} $i = I+1$;} \\
h_i, &\mbox{{\rm if} $ I+1 < i \le k+1$.}
\end{cases}
\end{equation*}
Then either there is exactly one greedy BST
of shape~$\lambda$ and type~$\alpha$ and no greedy BST
of shape~$\hat \lambda$ and type~$\alpha$, or
there is exactly one greedy BST
of shape~$\hat \lambda$ and type~$\alpha$ and no greedy BST
of shape~$\lambda$ and type~$\alpha$.
\end{lemma}

\begin{proof}

Informally, the main idea is that,
depending on whether the $I$th principal hook
covers the $I$th long overhang or the $I$th short overhang
(and this in turn depends on
whether the 1st long overhang
is the arm overhang or the leg overhang),
either there will be a unique way to place all the border strips,
or the $I$th border strip
will disconnect the remainder of the shape
in such a way that it cannot be completed to a full BST.
See for example Figure~\ref{fig:case0}, where $I=4$.

\begin{figure}[!ht]
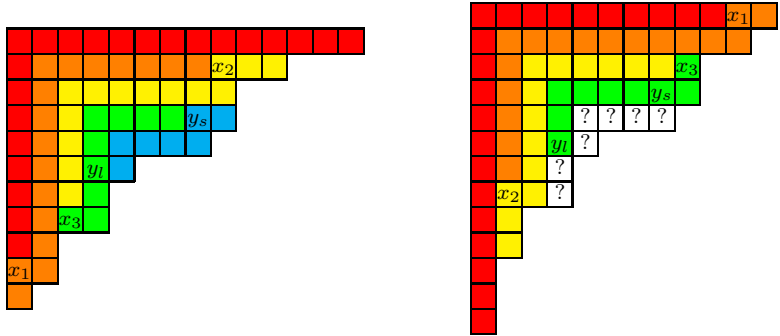

\begin{center}
\begin{ytableau}
*(red) & *(red) & *(red) & *(red) & *(red) & *(red) & *(red) & *(red)
   & *(red) & *(red) & *(red) & *(red) & *(red) & *(red) \\
*(red) & *(orange) & *(orange) & *(orange) & *(orange) & *(orange)
   & *(orange) & *(orange) & *(yellow) x_2 & *(yellow) & *(yellow) \\
*(red) & *(orange) & *(yellow) & *(yellow) & *(yellow) & *(yellow)
   & *(yellow) & *(yellow) & *(yellow)    \\
*(red) & *(orange) & *(yellow)    & *(green) & *(green) & *(green) & *(green) & *(cyan) y_s & *(cyan) \\
*(red) & *(orange) & *(yellow)    & *(green) & *(cyan) & *(cyan) & *(cyan) & *(cyan) \\
*(red) & *(orange) & *(yellow)    & *(green) y_l & *(cyan)  \\
*(red) & *(orange) & *(yellow)    & *(green) \\
*(red) & *(orange) & *(green) x_3 & *(green) \\
*(red) & *(orange)    \\
*(orange) x_1 & *(orange) \\
*(orange) 
\end{ytableau}
\hskip 0.5in
\begin{ytableau}
*(red) & *(red) & *(red) & *(red) & *(red) & *(red) & *(red) & *(red)
   & *(red) & *(red) & *(orange) x_1 & *(orange) \\
*(red) & *(orange) & *(orange) & *(orange) & *(orange) & *(orange)
   & *(orange) & *(orange) & *(orange) & *(orange) & *(orange) \\
*(red) & *(orange) & *(yellow) & *(yellow) & *(yellow) & *(yellow)
   & *(yellow) & *(yellow) & *(green) x_3   \\
*(red) & *(orange) & *(yellow) & *(green) & *(green) & *(green) & *(green) & *(green) y_s & *(green) \\
*(red) & *(orange) & *(yellow) & *(green) & ? & ? & ? & ? \\
*(red) & *(orange) & *(yellow) & *(green) y_l & ? \\
*(red) & *(orange) & *(yellow) & ? \\
*(red) & *(yellow) x_2 & *(yellow) & ? \\
*(red) & *(yellow) \\
*(red) & *(yellow) \\
*(red) \\
*(red) \\
*(red)
\end{ytableau}
\end{center}
\caption{Greedy Arrangement}
\label{fig:case0}
\end{figure}

On the left, the 4th (green) border strip
covers the 4th long overhang and there is a unique way
to complete the BST,
whereas on the right, the 4th border strip
disconnects the remainder of the shape,
rendering it impossible to cover with the single
remaining border strip of area~7.

More formally, in the Young diagrams of both $\lambda$ and $\hat\lambda$,
let $y_s$ be the box in the $I$th principal hook that is
adjacent to the $I$th short overhang, and
let $y_l$ be the box in the $I$th principal hook that is
adjacent to the $I$th long overhang.
In one of the two cases---the left-hand diagram
in Figure~\ref{fig:case0}---there will be a way to place
the $I$th border strip so that it
covers everything in the $I$th principal hook
up to but not including~$y_s$.
In effect, this is a ``greedy'' placement of the $I$th border strip.
Given this placement,
there is a unique way to place the $(I+1)$st border strip,
namely in such a way that it covers 
the $I$th short overhang, the box~$y_s$, and the
entire $(I+1)$st principal hook.
All subsequent border strips are then uniquely forced
to entirely cover the corresponding principal hooks.
On the other hand, if we attempt to place the $I$th border strip
``non-greedily'' then it must cover $y_s$,
but this will create an outer island that is disconnected from
the $(I+1)$st principal hook,
and there will be nowhere to place the $(I+1)$st border strip.

In the other case---the right-hand diagram in Figure~\ref{fig:case0}---the
$I$th border strip will cover both $y_s$ and~$y_l$,
as well as $b_I - a_I - 1$ boxes of the $I$th long overhang,
leaving $a_I + 1 > 0$ boxes in the $I$th long overhang uncovered.
We are now in trouble with the $(I+1)$st border strip,
because it must cover these $a_I + 1$ boxes
as well as the entire $(I+1)$st principal hook,
but this is impossible because these two nonempty regions
are connected only via the box~$y_l$,
which is occupied by the $I$th border strip.
So there are no BSTs in this case.

Note that it is in this final step of the argument
that we use the assumption that $I(\lambda) \le k$;
this implies that the $(I+1)$st principal hook is nonempty,
which we need for our contradiction.
\end{proof}

\subsubsection{The Case
\texorpdfstring{$\bm{I(\lambda) = k+1}$}{Ik1}} 
\label{subsec:Ik1}

As we observed in the proof of Lemma~\ref{lem:nongreedy1},
if we compare the greedy arrangements of $\lambda$ and~$\hat\lambda$,
then in one case an extreme end of the
1st, 2nd, 3rd, 4th, $\ldots$ border strips will be at the end of
an arm, leg, arm, leg, $\ldots$ respectively,
whereas in the other case an extreme end of the
1st, 2nd, 3rd, 4th, $\ldots$ border strips will be at the end of
a leg, arm, leg, arm, $\ldots$ respectively.
Since $I(\lambda)=k+1$, this process will continue until
all that is left uncovered is the final $(k+1)$st principal hook,
the box~$x_k$, and the adjacent $k$th overhang.
In one case, $x_k$ and the $k$th overhang will be attached
to the arm of the $(k+1)$st principal hook,
and in the other case, they will be attached
to the leg of the $(k+1)$st principal hook.
Motivated by this situation,
let us define an \emph{arm extension} of a hook
to be the shape obtained by attaching an additional row of boxes to the hook,
with the leftmost box of the new row directly above the box at
the top right corner of the hook.
Let us also define a \emph{leg extension} of a hook
to be the shape obtained by attaching an additional column of boxes to the hook,
with the topmost box of the new column directly to the left of the box
at the bottom left corner of the hook.
See Figure~\ref{fig:armlegextensions} for an example,
where the original hook is white and the additional row or column
has been colored red for visual clarity.

\begin{figure}[!ht]
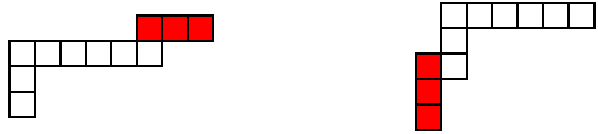

\begin{center}
\begin{ytableau}
\none & \none & \none & \none & \none & *(red) & *(red) & *(red) \\ 
 \    &       &       &       &       &        \\
 \    \\
 \
\end{ytableau}
\hskip 1truein
\begin{ytableau}
\none  &     &       &       &       &       &        \\
\none  &  \\
*(red) &  \\
*(red) \\
*(red)
\end{ytableau}
\end{center}
\caption{Arm and Leg Extensions of a Hook}
\label{fig:armlegextensions}
\end{figure}

Arm and leg extensions are defined even for \emph{improper} hooks;
i.e., hooks that consist of
just a single row or a single column. See Figure~\ref{fig:armlegextensions2}
for an example of the arm and leg extensions of an improper hook.

\begin{figure}[!ht]
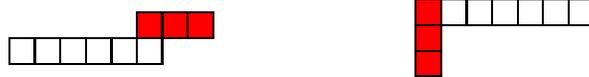

\begin{center}
\begin{ytableau}
\none & \none & \none & \none & \none & *(red) & *(red) & *(red) \\ 
 \    &       &       &       &       &        \\
\end{ytableau}
\hskip 1truein
\begin{ytableau}
*(red)  &     &       &       &       &       &        \\
*(red) \\
*(red)
\end{ytableau}
\end{center}
\caption{Arm and Leg Extensions of an Improper Hook}
\label{fig:armlegextensions2}
\end{figure}

The observations above show that
comparing the number of greedy BSTs
of type~$\alpha$ and shapes $\lambda$ and~$\hat\lambda$
reduces to comparing the number of BSTs of type~$\alpha^-$
of the arm and leg extensions\footnote{
We have not formally defined a BST of a shape
that is not a Young diagram,
but the definition is the obvious generalization.}
of the $(k+1)$st principal hook,
where $\alpha^-$ consists of the parts $\alpha_i$ for $i>k$.
Since $I(\lambda)=k+1$,
we may further assume that the $(k+1)$st principal hook
is \emph{not} self-conjugate.

\begin{lemma}

Assume that $I(\lambda) = k+1$.
For $1 \le i \le k$, let $\alpha_i = \alpha_i(\lambda)$
as defined by Equation~\eqref{eq:alphalambda}. Let
\begin{equation*}
\alpha_i :=
\begin{cases}
h_{k+1} - a_{k+1} + a_k, &\mbox{{\rm if} $i = k+1${\rm ;}} \\
a_{k+1} + 1, &\mbox{{\rm if} $i=k+2$.}
\end{cases}
\end{equation*}
Then the number of greedy BSTs of shape~$\lambda$
and type~$\alpha$ has the opposite parity from
the number of greedy BSTs of shape~$\hat\lambda$
and type~$\alpha$.
\end{lemma}

\begin{proof}

Let $H$ denote the $(k+1)$st principal hook.
Without loss of generality,
we may assume that the length of the first row of~$H$
exceeds the height of its first column.
For both the arm extension and the leg extension of~$H$,
we want to count the number of BSTs of type
\begin{equation*}
(h_{k+1} - a_{k+1} + a_k, a_{k+1} + 1).
\end{equation*}
In each case, the answer will be at most~2,
since the 2nd border strip must be placed at
either the upper right end or the lower left end
of the arm/leg extension;
it may be less than~2 because one or both of these putative placements
may violate the weakly increasing condition of a BST.

Note first that placing the 2nd border strip at the lower left end
of the arm extension of~$H$ is illegal,
because then it occupies precisely the entire 1st column of~$H$;
in particular, the top box of the 1st column will contain a~2,
and there will be boxes to its right that contain a~1.
On the other hand, placing the 2nd border strip at the upper right end
of the leg extension of~$H$ is legal;
since $a_{k+1} + 1 < b_{k+1} + 1$, the 2nd border strip will occupy
a proper subset of the first row of~$H$.
See Figure~\ref{fig:armlegtableaux} for an example
using the arm and leg extensions of Figure~\ref{fig:armlegextensions}.

\begin{figure}[!ht]
\begin{center}
\begin{ytableau}
\none & \none & \none & \none & \none & *(red)1 & *(red) 1 & *(red) 1 \\ 
 2    &   1   &   1   &   1   &   1   &    1   \\
 2    \\
 2
\end{ytableau}
\hskip 1truein
\begin{ytableau}
\none  &  1  &   1   &   1   &   2   &   2   &   2    \\
\none  &  1 \\
*(red) 1 &  1 \\
*(red) 1 \\
*(red) 1
\end{ytableau}
\end{center}
\caption{Illegal (left) and Legal (right) Tableaux}
\label{fig:armlegtableaux}
\end{figure}

Now we split into two cases, depending on whether $a_k = a_{k+1}$.
Suppose first that $a_k = a_{k+1}$.
Then placing the 2nd border strip at the upper right end
of the arm extension of~$H$ is illegal,
because then it occupies precisely the boxes in the additional row;
in particular, the leftmost additional box will contain a~2,
and there will be a box just below it containing a~1.
Similarly, placing the 2nd border strip at the lower left end
of the leg extension of~$H$ is illegal,
because then it occupies precisely the boxes in the additional column;
in particular, the topmost additional box will contain a~2,
and there will be a box just to the right of it containing a~1.
See Figure~\ref{fig:armlegtableaux2}.

\begin{figure}[!ht]
\begin{center}
\begin{ytableau}
\none & \none & \none & \none & \none & *(red) 2 & *(red) 2 & *(red) 2 \\ 
 1    &   1   &   1   &   1   &   1   &    1   \\
 1    \\
 1
\end{ytableau}
\hskip 1truein
\begin{ytableau}
\none  &  1  &   1   &   1   &   1   &   1   &   1    \\
\none  &  1 \\
*(red) 2 &  1 \\
*(red) 2 \\
*(red) 2
\end{ytableau}
\end{center}
\caption{Illegal Tableaux}
\label{fig:armlegtableaux2}
\end{figure}
Combining this observation with the previous observation,
we see that the arm extension admits 0 BSTs
while the leg extension admits 1 BST,
and 0 and~1 have opposite parity, as required.

Now suppose that $a_k \ne a_{k+1}$.  Then placing the 2nd border strip
at the upper right end of the arm extension will be legal,
because it will occupy either a proper subset of the additional row,
or it will occupy the entire additional row plus a proper subset of the
first row of~$H$ (it cannot occupy the entire first row of~$H$
because the first row of~$H$ has $b_{k+1} + 1 > a_{k+1} + 1$ boxes,
which is the area of the 2nd border strip,
and the 2nd border strip must cover at least one box in the additional row),
and either way, there is no violation of the
weakly increasing property.
Similarly, placing the 2nd border strip at the lower left end
of the leg extension will be legal.
So in this case, the arm extension admits 1 BST
while the leg extension admits 2 BSTs,
and 1 and~2 have opposite parity, as required.
\end{proof}

\subsubsection{The Case
\texorpdfstring{$\bm{I(\lambda)=\infty}$}{Iinfty}}
\label{subsec:Iinfty}

If $I(\lambda)=\infty$, then there is no second imbalance,
and it is not hard to see that this implies that
if the 1st principal hook is deleted from~$\lambda$,
then what remains is self-conjugate.
In particular, $\lambda' = \hat \lambda$.
We need the following standard fact.

\begin{lemma}
\label{lem:tensorsign}

For any $\lambda$ and~$\pi$,
\begin{equation*}
\chi_{\lambda}(\pi) = (\sgn\pi) \chi_{\lambda'}(\pi).
\end{equation*}
\end{lemma}

\begin{proof}

Textbooks typically prove this fact by noting that taking the
conjugate of a partition corresponds to tensoring with the sign
representation, but as pointed out to us by Richard Stanley,
it can be easily proved directly from the Murnaghan--Nakayama rule
as follows.
It is not hard to see that a border strip of even area
must have either an odd number of rows and an even number of columns,
or an even number of rows and an odd number of columns;
similarly, a border strip with odd area
must have either an odd number of rows and an odd number of columns,
or an even number of rows and an even number of columns.
It follows that if we transpose a BST of shape~$\lambda$,
then in the resulting BST of shape~$\lambda'$,
the signs of the border strips with even area will reverse
while the signs of the border strips with odd area
will remain the same.
Now, in the disjoint-cycle decomposition of~$\pi$,
a cycle of odd length is an even permutation,
and a cycle of even length is an odd permutation.
Thus the overall sign of the BST will reverse
if and only if $\pi$ has an odd number of
cycles of even length; i.e.,
if and only if $\pi$ is an odd permutation.
Since this argument applies for every BST, regardless of type,
the lemma follows.
\end{proof}

Lemma~\ref{lem:tensorsign} implies that
any odd permutation~$\pi$
such that $\chi_\lambda(\pi) \ne 0$
has the desired property that
$\chi_\lambda(\pi) \ne \chi_{\hat\lambda}(\pi)$.
\relax From this point on, we focus on finding such a~$\pi$.

\begin{lemma}
\label{lem:exactlyone}

Assume that $I(\lambda)=\infty$.
For $1 \le i \le k$, let $\alpha_i = \alpha_i(\lambda)$
as defined by Equation~\eqref{eq:alphalambda},
and let $\alpha_{k+1} = h_{k+1} + a_k + 1$.
Then there is exactly one BST of shape~$\lambda$
and type~$\alpha$.
\end{lemma}

\begin{proof}

There certainly exists a unique greedy BST
of shape~$\lambda$ and type~$\alpha$,
whose $(k+1)$st border strip covers all the boxes
not covered by the first $k$ border strips.
The point is that there cannot be any non-greedy BSTs,
because then the first $k$ border strips would leave uncovered
at least two disconnected components---an
outer island, and some boxes in the $(k+1)$st
principal hook---which therefore
cannot both be covered by the single remaining border strip.
\end{proof}

If the~$\alpha$ described in Lemma~\ref{lem:exactlyone} is
the cycle type of an odd permutation, then we are done,
so let us assume the contrary.
Note that breaking any single cycle of an even permutation
into two nonempty cycles yields an odd permutation.
In particular, any composition~$\beta$ with exactly $k+2$ nonzero parts
where $\beta_i = \alpha_i(\lambda)$
(as defined by Equation~\eqref{eq:alphalambda}) for $i\le k$,
and whose last two parts $\beta_{k+1}$ and $\beta_{k+2}$
sum to $h_{k+1} + a_k + 1$,
is the cycle type of an odd permutation.
If we can show that
the number of BSTs of type~$\beta$ is odd,
then that will imply that $\chi_\lambda(\beta) \ne 0$,
regardless of the signs of the BSTs.

The following lemma tells us that there are some strong
constraints on what a non-greedy BST can look like.

\begin{lemma}
\label{lem:nongreedy2}

Assume that $I(\lambda)=\infty$.
For $1 \le i \le k$, let $\alpha_i = \alpha_i(\lambda)$
as defined by Equation~\eqref{eq:alphalambda},
and assume that $\alpha$ has exactly $k+2$ nonzero parts.
Let $T$ be a non-greedy BST of shape~$\lambda$ and type~$\alpha$.
Let $i\le k$ be the smallest number such that
the $i$th border strip is \emph{not} positioned greedily.
Then for $i < i' \le k$,
\begin{equation}
\label{eq:toolarge}
h_{i'} - a_{i'} + a_{i'-1} > a_i + 1
\end{equation}
and
\begin{equation}
\label{eq:equality}
a_{i'} = a_{i'-1}.
\end{equation}
Moreover, $\alpha_{k+1} = h_{k+1}$ or $\alpha_{k+2} = h_{k+1}$.
\end{lemma}

\begin{proof}

The $i$th border strip of~$T$ is the $j$th slide for some~$j>0$;
we use $j$ for this number in the rest of this proof.

By Lemma~\ref{lem:principaldiag},
a border strip cannot contain more than one box
on the principal diagonal.
Here we have $k+2$ border strips and $k+1$ boxes on the principal diagonal,
so there can be at most one \emph{non-principal} border strip
(i.e., a border strip that does not contain a box on the principal diagonal).
In particular, there can be at most one border strip that
lies entirely in an overhang.

It follows that the inner island of~$T$ must be empty,
or else the inner and outer islands would both contain
non-principal border strips.
Thus the $i$th border strip must be slid as far as possible.
i.e., $j=a_i+1$.

For brevity, call the outer island~$S$.
The area of~$S$ is $j = a_i + 1$.
Equations \eqref{eq:toolarge} and~\eqref{eq:equality}
are vacuously true if $i=k$, so assume that $i<k$.
We prove the two equations jointly, by induction on~$i'$.
The base case is $i'=i+1$.
By Lemma~\ref{lem:hkak3},
\begin{equation*}
h_{i+1} - a_{i+1} \ge 3 > 1,
\end{equation*}
so $h_{i+1} - a_{i+1} + a_i > a_i + 1$,
proving Equation~\eqref{eq:toolarge} for $i' = i+1$.
But what Equation~\eqref{eq:toolarge} says is that
the $(i+1)$st border strip is too large to fit inside~$S$;
therefore, it must lie inside the $(i+1)$st principal hook.
In fact, it must entirely fill the $(i+1)$st principal hook;
the reason is that the area of the $(i+1)$st border strip is
\begin{equation*}
h_{i+1}-a_{i+1} + a_i \ge h_{i+1} - a_{i+1},
\end{equation*}
so any boxes in the $(i+1)$st principal hook not covered by
the $(i+1)$st border strip must lie in an
$(i+1)$st overhang, creating another region 
that must be covered by a non-principal border strip.
Therefore $h_{i+1} - a_{i+1} + a_i = h_{i+1}$,
proving $a_{i'} = a_{i'-1}$ for $i'=i+1$.

The proof of the induction step is similar to the proof of the base case.
Lemma~\ref{lem:hkak3} implies $h_{i'} - a_{i'} > 1$,
and by induction we may assume that $a_{i'-1} = a_i$,
so this proves Equation~\eqref{eq:toolarge}.
The $i'$th border strip cannot fit inside~$S$
and hence must lie inside the $i'$th principal hook.
Furthermore, by the same reasoning as we gave above,
it must entirely fill the $i'$th principal hook,
proving that $a_{i'} = a_{i'-1}$.

We have now almost completely specified the structure of~$T$.
The border strips before the $i$th are placed greedily;
the $i$th border strip is slid as far as possible;
the remaining border strips up to the $k$th border strip
each entirely occupy the respective principal hook.
The only boxes not covered by the first $k$ border strips
are in two disconnected components, namely $S$
and the $(k+1)$st principal hook.
Therefore $\alpha_{k+1}$ and $\alpha_{k+2}$
must specify the areas of these two regions,
and in particular one of them must equal~$h_{k+1}$.
\end{proof}

It turns out that for most shapes~$\lambda$ with $I(\lambda)=\infty$,
we can find an~$\alpha$ for which there exists a unique greedy BST
of shape~$\lambda$ and type~$\alpha$ and no non-greedy BST
of shape~$\lambda$ and type~$\alpha$; see Lemma~\ref{lem:no_nongreedy}.
However, there are a few exceptions
when the innermost hooks are very small;
these are dealt with in Lemma~\ref{lem:exceptional}.

\begin{lemma}
\label{lem:no_nongreedy}

Assume that $I(\lambda)=\infty$.
For $1 \le i \le k$, let $\alpha_i = \alpha_i(\lambda)$
as defined by Equation~\eqref{eq:alphalambda}.
Let
\begin{equation*}
\alpha_{k+2} :=
\begin{cases}
2, & \mbox{{\rm if} $a_k \ge 2$ {\rm and} $a_{k+1}=0${\rm ;}} \\
a_{k+1}+1, & \mbox{{\rm if} $a_k \ne a_{k+1}$ {\rm and} $a_{k+1} \ge 1${\rm ;}} \\
a_{k+1}+2, & \mbox{{\rm if} $a_k = a_{k+1} \ge 2$.}
\end{cases}
\end{equation*}
Let $\alpha_{k+1} := h_{k+1} + a_k + 1 - \alpha_{k+2}$.
Then there is exactly $1$ BST of shape~$\lambda$ and type~$\alpha$,
and this BST is greedy.
\end{lemma}

\begin{proof}

First, it is easy to verify that neither $\alpha_{k+1}$ nor $\alpha_{k+2}$
is equal to $h_{k+1}$, so by Lemma~\ref{lem:nongreedy2},
there can be no non-greedy BST of shape~$\lambda$ and type~$\alpha$.
So we just need to verify that there is exactly one greedy BST
of shape~$\lambda$ and type~$\alpha$.
To do this, it suffices to consider the residual shape left uncovered by
the greedy arrangement, and show that there is a unique BST of this
residual shape of type $(\alpha_{k+1}, \alpha_{k+2})$.

 In the case that $a_k \ge 2$ and $a_{k+1}=0$,
the residual shape is an arm extension
(or a leg extension, but we may assume without loss of generality
that it is an arm extension;
this is also true for the remaining cases,
so we will assume arm extensions in the rest of this proof without
further comment)
of a single box, where the additional row has length $a_k+1\ge 3$.
In this case, $\alpha_{k+2} = 2$,
so putting the 2nd border strip at the lower left end is illegal,
while putting the 2nd border strip at the upper right end is legal
(since the first row has length at least~$3$).
See Figure~\ref{fig:ak10}.
\bigskip

\begin{figure}[!ht]
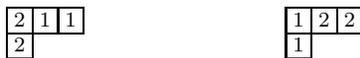

\begin{center}
\begin{ytableau}
 2 & 1 & 1 \\
 2
\end{ytableau}
\hskip 1truein
\begin{ytableau}
 1 & 2 & 2 \\
 1
\end{ytableau}
\end{center}
\caption{Illegal (left) and Legal (right) Tableaux}
\label{fig:ak10}
\end{figure}

In the case that $a_k \ne a_{k+1}$ and $a_{k+1} \ge 1$,
the residual shape is an arm extension of
a proper self-conjugate hook~$H$,
where the number of boxes in the additional row
is nonzero and is different from the number of boxes
in the first row (equivalently, the first column) of~$H$.
In this case, $\alpha_{k+2} = a_{k+1} + 1$,
which is exactly the number of boxes in the first column of~$H$.
Therefore putting the 2nd border strip at the lower left end
is illegal, since it would fill up the entire first column of~$H$,
resulting in a $2$ to the left of a~$1$ in the first row of~$H$.
On the other hand, putting the 2nd border strip at the
upper right end is legal,
because the 2nd border strip will either be a proper subset
of the additional row,
or it will contain the entire additional row as well as
a nonzero number of boxes at the right end of the first row of~$H$
(but not all of the first row of~$H$),
and either way, the danger of having a $2$ above a~$1$ will be
averted.
See Figures \ref{fig:ak_lt_ak1} and~\ref{fig:ak_gt_ak1}
for examples in which $a_k < a_{k+1}$ and $a_{k+1} < a_k$
respectively.

\begin{figure}[!ht]
\begin{center}
\begin{ytableau}
\none & \none & \none & 1 & 1 & 1 \\
  2   &   1   &   1   & 1 \\
  2 \\
  2 \\
  2
\end{ytableau}
\hskip 1truein
\begin{ytableau}
\none & \none & \none & 2 & 2 & 2 \\
  1   &   1   &   1   & 2 \\
  1 \\
  1 \\
  1
\end{ytableau}
\end{center}
\caption{Illegal (left) and Legal (right) Tableaux}
\label{fig:ak_lt_ak1}
\end{figure}

\begin{figure}[!ht]
\begin{center}
\begin{ytableau}
\none & \none & 1 & 1 & 1 & 1\\
  2   &   1   & 1 \\
  2 \\
  2
\end{ytableau}
\hskip 1truein
\begin{ytableau}
\none & \none & 1 & 2 & 2 & 2\\
  1   &   1   & 1 \\
  1 \\
  1
\end{ytableau}
\end{center}
\caption{Illegal (left) and Legal (right) Tableaux}
\label{fig:ak_gt_ak1}
\end{figure}

There remains the case that $a_k = a_{k+1} \ge 2$,
which is similar to the preceding case except that now
the additional row is exactly the same length
as the first row of~$H$.
In this case, $\alpha_{k+1} = a_{k+1} + 2$,
which is one more than the height of the first column of~$H$.
Putting the 2nd border strip at the lower left end
fills up the first column of~$H$ as well as an extra box
in the first row of~$H$;
since the first row of~$H$ has length at least~$3$,
there will be a~$1$ in the first row of~$H$ to the right of a~$2$,
which is illegal.
On the other hand, putting the 2nd border strip at the upper right end
fills up the additional row as well as
the rightmost box in the first row of~$H$,
and this is legal.  See Figure~\ref{fig:ak_eq_ak1}.

\begin{figure}[!ht]
\begin{center}
\begin{ytableau}
\none & \none & 1 & 1 & 1 \\
  2   &   2   & 1 \\
  2 \\
  2
\end{ytableau}
\hskip 1truein
\begin{ytableau}
\none & \none & 2 & 2 & 2\\
  1   &   1   & 2 \\
  1 \\
  1
\end{ytableau}
\end{center}
\caption{Illegal (left) and Legal (right) Tableaux}
\label{fig:ak_eq_ak1}
\end{figure}
This completes the proof.
\end{proof}

 Finally, we deal with the exceptional cases not
covered by Lemma~\ref{lem:no_nongreedy}.

\begin{lemma}
\label{lem:exceptional}

Assume that $I(\lambda)=\infty$.
For $1 \le i \le k$, let $\alpha_i = \alpha_i(\lambda)$
as defined by Equation~\eqref{eq:alphalambda},
and let $\alpha_{k+1} = h_{k+1}$ and $\alpha_{k+2} = a_k + 1$.
Let $G$ be the number of greedy BST
of shape~$\lambda$ and type~$\alpha$
and let $N$ be the number of non-greedy BST
of shape~$\lambda$ and type~$\alpha$.
\begin{enumerate}
\item \label{infty1} 
If $a_k = 0$ and $a_{k+1} = 0$ then $G=1$ and $N=2$.
\item \label{infty2} 
If $a_k = 1$ and $a_{k+1} = 0$ then $G=0$ and $N=1$.
\item \label{infty3} 
If $a_k = 1$ and $a_{k+1} = 1$ then $G=0$ and $N=1$.
\end{enumerate}
\end{lemma}

\begin{proof}

Let us consider Case~\ref{infty1} first.
The greedy arrangement leaves just two adjacent boxes uncovered,
and there is obviously only one way to cover them with
the (singleton) $(k+1)$st and $(k+2)$nd border strips.
See Figure~\ref{fig:infty1} for an example.

\begin{figure}[!ht]
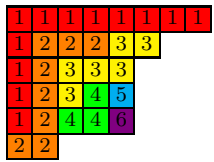

\begin{center}
\begin{ytableau}
 *(red) 1 & *(red) 1    & *(red)    1 & *(red)    1 & *(red)    1 & *(red)    1 & *(red) 1 & *(red) 1 \\
 *(red) 1 & *(orange) 2 & *(orange) 2 & *(orange) 2 & *(yellow) 3 & *(yellow) 3 \\
 *(red) 1 & *(orange) 2 & *(yellow) 3 & *(yellow) 3 & *(yellow) 3 \\
 *(red) 1 & *(orange) 2 & *(yellow) 3 & *(green)  4 & *(cyan)   5 \\
 *(red) 1 & *(orange) 2 & *(green)  4 & *(green)  4 & *(violet) 6 \\
 *(orange) 2 & *(orange) 2
\end{ytableau}
\end{center}
\caption{Case \ref{infty1}: $h_{k+1} = 1$ and $a_k = 0$}
\label{fig:infty1}
\end{figure}

Now consider the non-greedy BSTs.
Let $i'$ be the smallest integer with the property that $a_{i'}=a_l = 0$
for all $i'\le l \le k$,
and let the $i$th border strip be the first border strip
that is not positioned greedily.
We claim that $i=i'$.
Note first that Lemma~\ref{lem:nongreedy2} implies that $i\ge i'$.
If $i>i'$, then the outer island consists of the singleton box~$x_{i-1}$,
which cannot be covered by the
$(k+1)$st or $(k+2)$nd border strip
since it would violate the weakly increasing property.
See Figure~\ref{fig:violate1} for an example
with $4 = i > i' = 3$.

\begin{figure}[!ht]
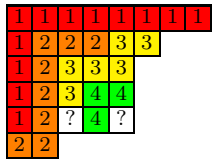

\begin{center}
\begin{ytableau}
 *(red) 1 & *(red) 1    & *(red)    1 & *(red)    1 & *(red)    1 & *(red)    1 & *(red) 1 & *(red) 1 \\
 *(red) 1 & *(orange) 2 & *(orange) 2 & *(orange) 2 & *(yellow) 3 & *(yellow) 3 \\
 *(red) 1 & *(orange) 2 & *(yellow) 3 & *(yellow) 3 & *(yellow) 3 \\
 *(red) 1 & *(orange) 2 & *(yellow) 3 & *(green)  4 & *(green)  4 \\
 *(red) 1 & *(orange) 2 & ?           & *(green)  4 & ? \\
 *(orange) 2 & *(orange) 2
\end{ytableau}
\end{center}
\caption{Case \ref{infty1}: Illegal Non-Greedy BST}
\label{fig:violate1}
\end{figure}

Conversely, if $i=i'$, then since $a_i = 0$,
the only possible slide of the $i$th border strip is the 1st slide,
so the outer island consists of
a singleton box at the extreme end of an $(i-1)$st overhang.
This box and the (singleton) $(k+1)$st hook
can be covered by the 
$(k+1)$st and $(k+2)$nd border strip in either order,
while the intervening border strips are uniquely forced
to entirely cover their corresponding principal hooks.
We obtain a total of 2 non-greedy BSTs.  See Figure~\ref{fig:eitherorder}
for an example.

\begin{figure}[!ht]
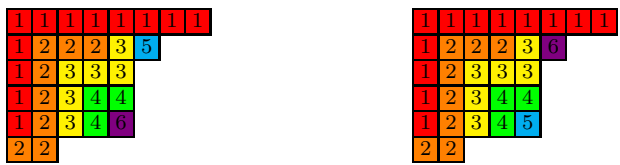

\begin{center}
\begin{ytableau}
 *(red) 1 & *(red) 1    & *(red)    1 & *(red)    1 & *(red)    1 & *(red)    1 & *(red) 1 & *(red) 1 \\
 *(red) 1 & *(orange) 2 & *(orange) 2 & *(orange) 2 & *(yellow) 3 & *(cyan) 5 \\
 *(red) 1 & *(orange) 2 & *(yellow) 3 & *(yellow) 3 & *(yellow) 3 \\
 *(red) 1 & *(orange) 2 & *(yellow) 3 & *(green)  4 & *(green)  4 \\
 *(red) 1 & *(orange) 2 & *(yellow) 3 & *(green)  4 & *(violet) 6 \\
 *(orange) 2 & *(orange) 2
\end{ytableau}
\hskip 1truein
\begin{ytableau}
 *(red) 1 & *(red) 1    & *(red)    1 & *(red)    1 & *(red)    1 & *(red)    1 & *(red) 1 & *(red) 1 \\
 *(red) 1 & *(orange) 2 & *(orange) 2 & *(orange) 2 & *(yellow) 3 & *(violet) 6 \\
 *(red) 1 & *(orange) 2 & *(yellow) 3 & *(yellow) 3 & *(yellow) 3 \\
 *(red) 1 & *(orange) 2 & *(yellow) 3 & *(green)  4 & *(green)  4 \\
 *(red) 1 & *(orange) 2 & *(yellow) 3 & *(green)  4 & *(cyan) 5 \\
 *(orange) 2 & *(orange) 2
\end{ytableau}
\end{center}
\caption{Case \ref{infty1}: Two Non-Greedy BSTs}
\label{fig:eitherorder}
\end{figure}

Now consider Case~\ref{infty2}.
The greedy arrangement leaves uncovered a set of three boxes
(two in one row and one in another),
and there is no legal way to cover this with
a $(k+1)$st border strip of area~1 and
a $(k+2)$nd border strip of area~2.
See Figure~\ref{fig:violate2} for an example.

\begin{figure}[!ht]
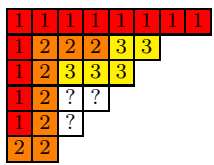

\begin{center}
\begin{ytableau}
 *(red) 1 & *(red) 1    & *(red)    1 & *(red)    1 & *(red)    1 & *(red)    1 & *(red) 1 & *(red) 1 \\
 *(red) 1 & *(orange) 2 & *(orange) 2 & *(orange) 2 & *(yellow) 3 & *(yellow) 3 \\
 *(red) 1 & *(orange) 2 & *(yellow) 3 & *(yellow) 3 & *(yellow) 3 \\
 *(red) 1 & *(orange) 2 & ? & ? \\
 *(red) 1 & *(orange) 2 & ? \\
 *(orange) 2 & *(orange) 2
\end{ytableau}
\end{center}
\caption{Case~\ref{infty2}: Illegal Greedy BST}
\label{fig:violate2}
\end{figure}

Now consider the non-greedy BSTs.
Let $i'$ be the smallest integer with the property that $a_{i'}=a_l = 1$
for all $ i' \le l \le k$,
and let the $i$th border strip be the first border strip
that is not positioned greedily.
The $i$th border strip must be slid as far as possible,
since otherwise we would have both an outer island
and an inner island,
creating more disconnected components
than we have border strips to cover them.
We claim that $i=i'$.
The argument is similar to the argument in Case~\ref{infty1}.
Lemma~\ref{lem:nongreedy2} implies that $i\ge i'$.
If $i > i'$, then because $a_i = 1$,
the outer island consists of two boxes,
one of which is~$x_{i-1}$,
which cannot be covered by the $(k+1)$st
or $(k+2)$nd border strip without violating
the weakly increasing property.
See Figure~\ref{fig:violate3}
for an example with $3 = i > i' = 2$.

\begin{figure}[!ht]
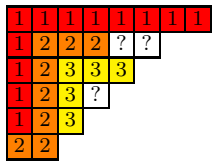

\begin{center}
\begin{ytableau}
 *(red) 1 & *(red) 1    & *(red)    1 & *(red)    1 & *(red)    1 & *(red)    1 & *(red) 1 & *(red) 1 \\
 *(red) 1 & *(orange) 2 & *(orange) 2 & *(orange) 2 & ? & ? \\
 *(red) 1 & *(orange) 2 & *(yellow) 3 & *(yellow) 3 & *(yellow) 3 \\
 *(red) 1 & *(orange) 2 & *(yellow) 3 & ? \\
 *(red) 1 & *(orange) 2 & *(yellow) 3 \\
 *(orange) 2 & *(orange) 2
\end{ytableau}
\end{center}
\caption{Case~\ref{infty2}: Illegal Non-Greedy BST}
\label{fig:violate3}
\end{figure}

Conversely, if $i=i'$ then there exists a unique
non-greedy BST, because the outer island has area~2
and the innermost principal hook has area~1,
so the outer island must be covered by the $(k+2)$nd border strip
and the innermost principal hook must be covered by
the $(k+1)$st border strip.
It is readily checked that covering the outer island
with the $(k+2)$nd border strip does not violate
the weakly increasing property, regardless of
whether $a_{i-1} = 0$ or $a_{i-1} > 1$;
see Figure~\ref{fig:infty2} for diagrams of these two subcases.
\bigskip

\begin{figure}[!ht]
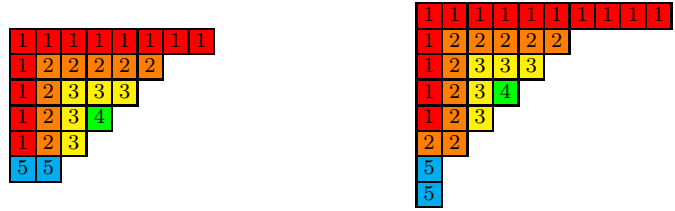

\begin{center}
\begin{ytableau}
 *(red) 1 & *(red) 1    & *(red)    1 & *(red)    1 & *(red)    1 & *(red)    1 & *(red) 1 & *(red) 1 \\
 *(red) 1 & *(orange) 2 & *(orange) 2 & *(orange) 2 & *(orange) 2 & *(orange) 2 \\
 *(red) 1 & *(orange) 2 & *(yellow) 3 & *(yellow) 3 & *(yellow) 3 \\
 *(red) 1 & *(orange) 2 & *(yellow) 3 & *(green)   4 \\
 *(red) 1 & *(orange) 2 & *(yellow) 3 \\
 *(cyan) 5 & *(cyan) 5
\end{ytableau}
\hskip 1truein
\begin{ytableau}
 *(red) 1 & *(red) 1    & *(red)    1 & *(red)    1 & *(red)    1
  & *(red)    1 & *(red) 1 & *(red) 1 & *(red)    1 & *(red)    1 \\
 *(red) 1 & *(orange) 2 & *(orange) 2 & *(orange) 2 & *(orange) 2 & *(orange) 2 \\
 *(red) 1 & *(orange) 2 & *(yellow) 3 & *(yellow) 3 & *(yellow) 3 \\
 *(red) 1 & *(orange) 2 & *(yellow) 3 & *(green)   4 \\
 *(red) 1 & *(orange) 2 & *(yellow) 3 \\
 *(orange) 2 & *(orange) 2 \\
 *(cyan) 5 \\
 *(cyan) 5
\end{ytableau}
\end{center}
\caption{Case~\ref{infty2}: $h_{k+1}=1$ and $a_k=1$}
\label{fig:infty2}
\end{figure}

Finally, there is Case~\ref{infty3},
which is very similar to Case~\ref{infty2},
so we will just sketch the argument.
The greedy arrangement leaves 5 boxes uncovered,
which cannot be covered by
a $(k+1)$st border strip of area~3 and
a $(k+2)$nd border strip of area~2;
see Figure~\ref{fig:violate4}.

\begin{figure}[!ht]
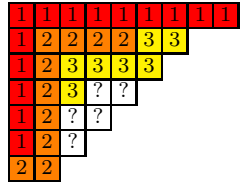

\begin{center}
\begin{ytableau}
 *(red) 1 & *(red) 1    & *(red)    1 & *(red)    1 & *(red)    1
  & *(red)    1 & *(red) 1 & *(red) 1 & *(red) 1 \\
 *(red) 1 & *(orange) 2 & *(orange) 2 & *(orange) 2 & *(orange) 2 & *(yellow) 3 & *(yellow) 3 \\
 *(red) 1 & *(orange) 2 & *(yellow) 3 & *(yellow) 3 & *(yellow) 3 & *(yellow) 3 \\
 *(red) 1 & *(orange) 2 & *(yellow) 3 & ?  & ? \\
 *(red) 1 & *(orange) 2 & ? & ? \\
 *(red) 1 & *(orange) 2 & ?  \\
 *(orange) 2 & *(orange) 2
\end{ytableau}
\end{center}
\caption{Case~\ref{infty3}: Illegal Greedy BST}
\label{fig:violate4}
\end{figure}

For the non-greedy BSTs,
we define $i$ and~$i'$ as in Case~\ref{infty2}, and argue as before
that we must have $i=i'$; see Figure~\ref{fig:violate5}.

\begin{figure}[!ht]
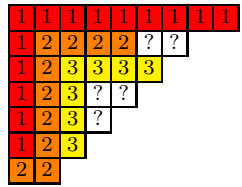

\begin{center}
\begin{ytableau}
 *(red) 1 & *(red) 1    & *(red)    1 & *(red)    1 & *(red)    1
  & *(red)    1 & *(red) 1 & *(red) 1 & *(red) 1 \\
 *(red) 1 & *(orange) 2 & *(orange) 2 & *(orange) 2 & *(orange) 2 & ? & ? \\
 *(red) 1 & *(orange) 2 & *(yellow) 3 & *(yellow) 3 & *(yellow) 3 & *(yellow) 3 \\
 *(red) 1 & *(orange) 2 & *(yellow) 3 & ?  & ? \\
 *(red) 1 & *(orange) 2 & *(yellow) 3 & ? \\
 *(red) 1 & *(orange) 2 & *(yellow) 3  \\
 *(orange) 2 & *(orange) 2
\end{ytableau}
\end{center}
\caption{Case~\ref{infty3}: Illegal Non-Greedy BST}
\label{fig:violate5}
\end{figure}

If $i=i'$ then as in Case~\ref{infty2} we argue that
there is a unique non-greedy BST,
with the $(k+1)$st border strip covering the innermost principal hook
and the $(k+2)$nd border strip covering the outer island;
see Figure~\ref{fig:infty3}.

\begin{figure}[!ht]
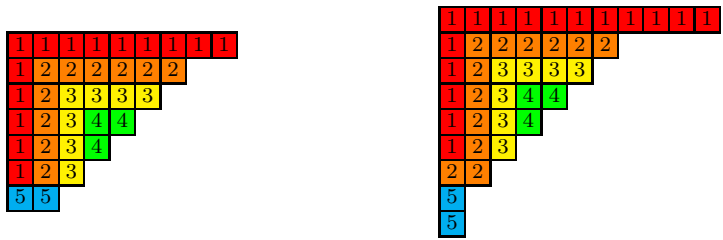
  
\begin{center}
\begin{ytableau}
 *(red) 1 & *(red) 1    & *(red)    1 & *(red)    1 & *(red)    1
  & *(red)    1 & *(red) 1 & *(red) 1 & *(red) 1 \\
 *(red) 1 & *(orange) 2 & *(orange) 2 & *(orange) 2 & *(orange) 2 & *(orange) 2 & *(orange) 2 \\
 *(red) 1 & *(orange) 2 & *(yellow) 3 & *(yellow) 3 & *(yellow) 3 & *(yellow) 3 \\
 *(red) 1 & *(orange) 2 & *(yellow) 3 & *(green) 4  & *(green) 4 \\
 *(red) 1 & *(orange) 2 & *(yellow) 3 & *(green) 4 \\
 *(red) 1 & *(orange) 2 & *(yellow) 3  \\
 *(cyan) 5 & *(cyan) 5
\end{ytableau}
\hskip 1truein
\begin{ytableau}
 *(red) 1 & *(red) 1    & *(red)    1 & *(red) 1 & *(red) 1
  & *(red)    1 & *(red) 1 & *(red) 1 & *(red) 1 & *(red) 1 & *(red) 1 \\
 *(red) 1 & *(orange) 2 & *(orange) 2 & *(orange) 2 & *(orange) 2 & *(orange) 2 & *(orange) 2 \\
 *(red) 1 & *(orange) 2 & *(yellow) 3 & *(yellow) 3 & *(yellow) 3 & *(yellow) 3 \\
 *(red) 1 & *(orange) 2 & *(yellow) 3 & *(green) 4  & *(green) 4 \\
 *(red) 1 & *(orange) 2 & *(yellow) 3 & *(green) 4 \\
 *(red) 1 & *(orange) 2 & *(yellow) 3  \\
 *(orange) 2 & *(orange) 2 \\
 *(cyan) 5 \\
 *(cyan) 5
\end{ytableau}
\end{center}
\caption{Case~\ref{infty3}: $h_{k+1} = 3$ and $a_k=1$}
\label{fig:infty3}
\end{figure}
In all cases, we have found an odd number of BSTs,
so the character value cannot be zero.
\end{proof}

\section{Concluding Remarks}

A crude upper bound on the number of queries needed
for our algorithm is $O(n^{3/2})$.
In the forward pass, there are at most $n^{1/2}$ principal hooks,
and the size of each principal hook can be determined in at most $n$ queries.
In the backward pass, there are again at most $n^{1/2}$ principal hooks,
and for each principal hook, determining $a_i$ and~$b_i$
requires at most $n$ queries, and distinguishing
doppelg\"angers requires a constant number of queries.
We expect that a more careful analysis, which we have not carried out,
will show that the required number of queries is (approximately) linear
in~$n$, because if there are a lot of principal hooks
then the number of queries per principal hook will be reduced.

It is natural to ask if a more efficient algorithm can be found.
As we mentioned in the Introduction,
empirically it seems that permutations that consist mostly of fixed points
are good distinguishers.
Enumerating the corresponding BSTs naturally leads
to enumerating skew tableaux,
for which there exist formulae
such as the Naruse hook-length formula
(see~\cite{morales-pak-panova} for a readable description and proof
of the Naruse hook-length formula).
Although it does not seem easy to prove that
various alternating sums of hook-length formulae
cannot coincide in value, perhaps it can be done.
If so, the number of queries needed could conceivably be drastically reduced.

\bibliography{refs}
\bibliographystyle{plain}

\end{document}